\documentclass[pdflatex,sn-apa]{sn-jnl}% APA Reference Style 
\usepackage{graphicx}%
\usepackage{multirow}%
\usepackage{amsmath,amssymb,amsfonts}%
\usepackage{amsthm}%
\usepackage{mathrsfs}%
\usepackage[title]{appendix}%
\usepackage{xcolor}%
\usepackage{textcomp}%
\usepackage{manyfoot}%
\usepackage{booktabs}%
\usepackage{algorithm}%
\usepackage{algorithmicx}%
\usepackage{algpseudocode}%
\usepackage{listings}%
\usepackage{mathenv_no_unicode}
\theoremstyle{plain}%
\newtheorem{theorem}{Theorem}%  meant for continuous numbers
\newtheorem{lemma}[theorem]{Lemma}% 

\theoremstyle{remark}%
\newtheorem{example}{Example}%

\theoremstyle{definition}%

\begin{document}

\title[The differential structure shared by probability and moment matching priors on non-regular statistical models via the Lie derivative]
{The differential structure shared by probability and moment matching priors on non-regular statistical models via the Lie derivative\footnote{This is a preprint submitted to \textit{Sankhya A}.}}


\author*{\fnm{Masaki} \sur{Yoshioka}}\email{yoshioka@sigmath.es.osaka-u.ac.jp}

\author{\fnm{Fuyuhiko} \sur{Tanaka}}\email{ftanaka.celas@osaka-u.ac.jp}


\affil{\orgname{The University of Osaka}, \orgaddress{\city{Osaka}, \country{Japan}}}

% \affil[2]{\orgdiv{Center for Education in Liberal Arts and Science}, \orgname{Osaka University}, \orgaddress{\street{Street}, \city{City}, \postcode{10587}, \state{State}, \country{Country}}}

% \affil[3]{\orgdiv{Department}, \orgname{Organization}, \orgaddress{\street{Street}, \city{City}, \postcode{610101}, \state{State}, \country{Country}}}

\abstract{
	In Bayesian statistics, the selection of noninformative priors is a crucial issue. There have been various discussions on theoretical justification, problems with the Jeffreys prior, and alternative objective priors. Among them, we focus on two types of matching priors consistent with frequentist theory: the probability matching priors and the moment matching priors. In particular, no clear relationship has been established between these two types of priors on non-regular statistical models, even though they share similar objectives.

	Considering information geometry on a one-sided truncated exponential family, a typical example of non-regular statistical models, we find that the Lie derivative along a particular vector field provides the conditions for both the probability and moment matching priors. Notably, this Lie derivative does not appear in regular models. These conditions require the invariance of a generalized volume element with respect to differentiation along the non-regular parameter. This invariance leads to a suitable decomposition of the one-sided truncated exponential family into one-dimensional submodels. This result promotes a unified understanding of probability and moment matching priors on non-regular models.
% 163words, (Please provide an abstract of 150 to 250 words.)
}

\keywords{Information geometry, Bayesian statistics, truncated exponential family, probability matching prior, moment matching prior}


\pacs[MSC Classification]{Primary 62F15; Secondary 62B11} % 62F15 Bayesian inference, 62B11 Information geometry (statistical aspects)

\maketitle

\section{Introduction}\label{sec:introduction}
	The choice of noninformative priors is one of the challenges in Bayesian statistics. Since the effectiveness of Bayesian methods depends on priors, it is necessary to set up an appropriate prior for statistical analysis according to statistical models and tasks. We can use a subjective prior if we know something about the parameters. However, ``objectivity'' in the prior often makes Bayesian methods effective when we do not know the parameters. Such a prior is called a noninformative prior or an objective prior.

Theoretical studies on noninformative priors have a long history and many objectivity criteria and resulting priors. \cite{jeffreys1961TheoryProbability} derives an invariant prior under parameter transformations, so-called the Jeffreys prior. The theory of reference priors justifies the Jeffreys prior \citep{bernardo1979ReferencePosteriorDistributions}. 
Furthermore, the probability matching prior and the moment matching prior are also well-known noninformative priors that combine the frequentist theory and the Bayesian statistical theory. 

The probability matching prior matches the posterior and frequentist probabilities of the confidence interval. \cite{welch1963FormulaeConfidencePoints} introduce this idea, which has been developed since then (see also \cite{peers1965ConfidencePointsBayesian}, \cite{tibshirani1989NoninformativePriorsOne}, \cite{ghosh1992NoninformativePriors}, \cite{datta2004ProbabilityMatchingPriors} and \cite{sweeting2008PredictiveProbabilityMatching}). Probability matching priors also have invariance under parameter transformations \citep{datta1996InvarianceNoninformativePriors}. On the other hand, the moment matching prior by \cite{ghosh2011MomentMatchingPriors} matches the Bayesian posterior mean and the maximum likelihood estimator. Using moment matching priors, the posterior mean has the asymptotic optimality of the maximum likelihood estimator, and the bias correction can also be performed.

On the other hand, there are also many studies of noninformative priors in the non-regular statistical models where the support of the distributions depends on the parameters (\cite{ghosal1997ReferencePriorsMultiparameter, ortega2016GeneralizationJeffreysRule, hashimoto2021PredictiveProbabilityMatching, shemyakin2023HellingerInformationMatrix}). Bayesian statistics for the non-regular models is also essential since these models have many applications \citep{lancaster1997ExactStructuralInference,brown1995StochasticSpecificationRandom}. In particular, the model by \cite{ghosal1995AsymptoticBehaviourBayes} has been investigated well. 
\cite{ghosal1999ProbabilityMatchingPriors} gives the probability matching prior, and \cite{hashimoto2019MomentMatchingPriors} provides the moment matching prior in this model.

We construct a theory that treats probability and moment matching priors in a unified manner in these kinds of non-regular models. 
Two matching priors have a similar purpose: choosing a prior distribution that matches the frequentist theory. However, their relationship could have been more precise, although they have been discussed separately. Therefore, by considering the information geometry of non-regular models, we clarify the structure of the two matching priors.

Information geometry is helpful in statistical theory \citep{amari2000MethodsInformationGeometry}, including figuring out noninformative priors. For regular models, \cite{takeuchi2005AlphaparallelPriorIts} give a family of noninformative priors called the $\alpha$-parallel priors from a geometric point of view. \cite{tanaka2023GeometricPropertiesNoninformative} discovers geometric properties of some noninformative priors and, in particular, clarifies that the conditions of the moment matching prior depend on the geometric properties. However, the use of information geometry in non-regular models is limited. Recently, \cite{yoshioka2023AlphaparallelPriorsOneSided} discuss information geometry of a one-sided truncated exponential family (oTEF), a typical non-regular model.

The present paper provides sufficient conditions and the characterization of the probability and moment matching priors for multivariate non-regular models, especially for an oTEF. In this model, we derive the asymptotic expansion of the posterior distribution and the partial differential equations for the two matching priors with nuisance parameters. Furthermore, by restricting the model to an oTEF, we represent those partial differential equations by the Riemannian metric and the $\alpha$-connection coefficients \citep{yoshioka2023AlphaparallelPriorsOneSided}. Then, the Lie derivative along the common vector field appears in the conditions of the two types of matching priors. These conditions require an invariance of a generalization of the volume element with respect to the differentiation with respect to the non-regular parameter under a parameter transformation. This geometric property induces a natural $-1$-dimensional submodel.

This paper is organized as follows. Section 2 defines the one-sided truncated family and the notations. In Section 3, we derive the conditions of the probability matching priors on the oTF. In Section 4, we also derive the conditions of the moment matching priors on the oTF. Then, we discuss the relationship between the two types of matching priors in Section 5. Finally, Section 6 summarizes these results.

\section{One-sided truncated family}\label{sec:preliminaries}
	A one-sided truncated family \citep{akahira2021MaximumLikelihoodEstimation}, shortly an oTF, is a typical non-regular statistical model with a parameter-dependent support. Let $\Theta$ be an open subset of $ \bR^d $ and $ I=\paren{ I_1,I_2 } $ be an open interval, where $ -\infty \leq I_{1}<I_{2}\leq \infty $.
Consider a parametrized family $ \model=\Set{ P_{\theta,\gamma}: \theta\in \Theta,\gamma \in I } $ of probability distributions $P_{\theta,\gamma}$, having a density
\begin{align}
	p(x;\theta,\gamma) & =  q(x;\theta)\e^{-\psi(\theta,\gamma)} \cdot \indi_{[\gamma,I_2)} (x) \quad \paren{x\in I}
	\label{eq:oTF-density}
\end{align}
with respect to the Lebesgue measure, where $ q(x;\theta) $ is positive. This family $ \model $ is called a \textit{one-sided truncated family (oTF)}, or more precisely, a \textit{lower-truncated family (lTF)}.
We call the parameter $ \gamma $ the \textit{truncation parameter}, and the parameter $ \theta=(\theta^{1},\ldots,\theta^{d}) $ the \textit{regular parameter}.
Suppose that $ \model $ is identifiable in the sense that for any $ \theta_{1}, \theta_{2}\in \Theta $ and $ \gamma \in I $, $ P_{\theta_{1},\gamma} = P_{\theta_{2},\gamma} $ implies $ \theta_{1} = \theta_{2} $.
We also assume that $ p(x;\theta,\gamma) $ is infinitely differentiable in $ \theta $ and $ \gamma $ on the interval $ (\gamma,I_{2}) $.
An oTF is a non-regular statistical model because the support $ [\gamma, I_{2}] $ of the distribution depends on the truncation parameter $ \gamma $.
Note that the submodel $ \Set{ P_{\theta,\gamma}:\theta\in \Theta } $ is regular for any $ \gamma\in I $.

We also consider a one-sided truncated exponential family, a submodel of an oTF. When the function $ q(x;\theta) $ has the form
\begin{align}
	q(x;\theta) & = \exp\braces{ \sum_{i=1}^d \theta^iF_i(x)+M(x) }\quad \paren{x\in I},
	\label{eq:oTEF-density}
\end{align}
where $ M\in C(I),\, F_i\in C^\infty(I)\,(i=1,\ldots,d)$, we call the family $ \model_{e}=\Set{ P_{\theta,\gamma}: \theta\in \Theta,\gamma \in I } $ a \textit{one-sided truncated exponential family (oTEF)}.
In this case, we call the regular parameters $ \theta $ \textit{natural parameters}.
\cite{bar-lev1984LargeSampleProperties}, \cite{akahira2016SecondOrderAsymptotic} and \cite{akahira2017StatisticalEstimationTruncated} investigate the statistical properties of the oTEF in detail. For any fixed parameter $ \gamma\in I $, the submodel $ \Set{ P_{\theta,\gamma}:\theta\in \Theta } $ is an exponential family.

Let $ X_{1},\ldots, X_{n} $ be i.i.d. random samples from a distribution $ p(x;\theta,\gamma) $ in an oTF $ \model $. The MLE of $ \gamma $ is given by
\begin{align*}
	\hat{\gamma}_{\ML} = \min_{1\leq i\leq n}X_{i}.
\end{align*}
Assume that there exists a unique solution $ \hat{\theta}_{\ML}^{i}\:(1 \leq i \leq n) $ of the likelihood equation
$ \paren{ 1/n } \sum_{j=1}^{n}\log p(x_{j};\theta,x_{(1)})=0 $ for any $ x=(x_{1},\ldots,x_{n})\in (\gamma,I_{2})^{n} $. Then, there exists an MLE $ \hat{\theta}_{\ML}^{i} $ that satisfies the likelihood equation \citep{akahira2021MaximumLikelihoodEstimation}
\begin{align*}
	\sum_{j=1}^{n}\pd_{i}\log p(X_{j};\theta,\hat{\gamma}_{\ML})=0.
\end{align*}
Note that the two MLEs $ \hat{\theta}_{\ML}^{i},\,\hat{\gamma}_{\ML} $ have different orders of convergence. The subsequent sections will provide further details on this difference. Writing the expectations of the derivatives of the log-likelihood functions in vector notation simplifies the presentation of results in the following sections. Define $ \pd_{i} \coloneqq \pd/\pd \theta^{i} $ for $ i=1,\ldots,d $ and $ \pd_{\gamma} \coloneqq \pd/\pd \gamma $. Let $ D_{\theta} \coloneqq \paren{ \pd_{1},\ldots,\pd_{d} }^{\top} $. We write the Kronecker product of two matrices $ A $ and $ B $ as $ A\otimes B $. The Kronecker product of $ r $ copies of matrix $ A $ is denoted by $ A^{\otimes r} $. We define
\begin{align}
	A^{\paren{ r,s }}(\theta,\gamma) & \coloneqq \Exp[D_{\theta}^{\otimes r} \paren{ \pd_{\gamma} }^{s} \log p(X_{1};\theta,\gamma)] \quad \paren{ r,s = 0,1,2, \ldots },                                               \\
	c(\theta,\gamma)                 & \coloneqq A^{\paren{ 0,1 }}(\theta,\gamma)= \Exp[\pd_{\gamma}\log p(X,\theta,\gamma)]= -\pd_{\gamma}\psi(\theta,\gamma),                                                         \\
	\hat{A}^{\paren{ r,s }}          & \coloneqq \frac{ 1 }{ n }\sum_{i=1}^{n} D_{\theta}^{\otimes r}\paren{ \pd_{\gamma} }^{s}\log p(X_{i},\hat{\theta}_{\ML},\hat{\gamma}_{\ML}) \quad \paren{ r,s = 0,1,2, \ldots }, \\
	\hat{c}                          & \coloneqq \frac{ 1 }{ n }\sum_{i=1}^{n}\pd_{\gamma}\log p(X_{i};\hat{\theta}_{\ML},\hat{\gamma}_{\ML}).
\end{align}
$ A^{\paren{ r,s }}(\theta,\gamma) $ and $ \hat{A}^{\paren{ r,s }} $ are $ d^{r} $-dimensional vectors.
Each component of $ A^{\paren{ r,s }}(\theta,\gamma) $ is written as
\begin{align*}
	A^{\paren{ 1,s }}_{i}(\theta,\gamma)   & \coloneqq \Exp[\pd_{i} \paren{ \pd_{\gamma} }^{s} \log p(X_{1};\theta,\gamma)],              \\
	A^{\paren{ 2,s }}_{ij}(\theta,\gamma)  & \coloneqq \Exp[\pd_{i}\pd_{j} \paren{ \pd_{\gamma} }^{s} \log p(X_{1};\theta,\gamma)],       \\
	A^{\paren{ 3,s }}_{ijk}(\theta,\gamma) & \coloneqq \Exp[\pd_{i}\pd_{j} \pd_{k}\paren{ \pd_{\gamma} }^{s} \log p(X_{1};\theta,\gamma)]
\end{align*}
for $ i,j,k=1,\ldots,d $ and $ s=0,1,2,\ldots $. We use similar notation for the components of $ \hat{A}^{\paren{ r,s }} $.
We sometimes omit the arguments as $ A^{(r,s)},\; c $ for $ A^{(r,s)}(\theta,\gamma),\;c(\theta,\gamma) $ when the arguments are clear from context.
Furthermore, we abbreviate the transposed vector $ \paren{ A^{(r,s)} }^{\top} $ as $ A^{(r,s)\top} $.

We introduce the notation of information geometry to examine the geometric aspects of the two matching priors. We use the Riemannian metric and the $ \alpha $-connections defined by \cite{yoshioka2023AlphaparallelPriorsOneSided} for the oTF $ \model $. The Riemannian metric $ g $ on $ \model $ is defined by
\begin{align}
	g_{ij}(\theta,\gamma)           & \coloneqq -\Exp[\pd_{i}\pd_{j}\log p(X_{1};\theta,\gamma)] \quad \paren{ i,j=1,\ldots,d }, \\
	g_{i\gamma}(\theta,\gamma)      & \coloneqq 0 \quad \paren{ i=1,\ldots,d },                                                  \\
	g_{\gamma\gamma}(\theta,\gamma) & \coloneqq \braces{ \pd_{\gamma}\psi(\theta,\gamma) }^{2}.
\end{align}
The submatrix $ g_{\theta}=\paren{ g_{ij}(\theta,\gamma) }_{1\leq i,j\leq d} $, consisting of the regular part of $ g $, is the Fisher information matrix of the submodel $ \Set{ P_{\theta,\gamma}:\theta\in \Theta } $ for $ \gamma \in I $.
Let $ \hat{g}_{\theta} $ denote a matrix consisting of $ \hat{A}^{\paren{ 2,0 }}_{ij} $ for $ i,j=\Ilist $ according to the relation $ g_{\theta} = \paren{ -A^{\paren{ 2,0 }}_{ij} }_{1\leq i,j\leq d} $.
Note that the full matrix $ \paren{ g_{ab} }\; \paren{ a,b=\Ilist* } $ corresponds to the asymptotic covariance of the vector $ \paren{ \qML,\cML } $ \citep{yoshioka2023InformationGeometricApproachOneSided}.
Let $\Gamma^{g} $ denote the Levi-Civita connection coefficients with respect to the metric $ g $, given by
\begin{align*}
	\Gamma^{g}_{ab,c}(\theta,\gamma) & = \frac{ 1 }{ 2 }\paren{ \pd_{a}g_{bc}(\theta,\gamma)+\pd_{b}g_{ca}(\theta,\gamma)-\pd_{c}g_{ab}(\theta,\gamma) }
\end{align*}
for $ a,b,c=\Ilist* $.
We also define the $ \alpha $-connections on $ \model $ with the connection coefficients
\begin{align*}
	\aChris{ab,c} (\theta,\gamma) & \coloneqq \alpha\Exp\bracket{ \paren{ \partial_a \partial_b l_{X_{1},\theta,\gamma} } \, \paren{ \partial_c l_{X_{1},\theta,\gamma} } }+(1 - \alpha)\Gamma^g_{ab,c}
\end{align*}
for $ \alpha \in \bR $ and $ a,b,c=\Ilist* $,
where $ l_{X_{1},\theta,\gamma}=\log p (X_{1};\theta,\gamma) $.
In particular, the symbol $ \aChris[1]{ab,c} $, also called the e-connection, is denoted by $\aChris[e]{ab,c}$.
Here, we ignore the null set $ \braces{ x=\gamma } $ where $ \log p(x;\theta,\gamma) $ is not differentiable in the above expectations.
Only the regular parts $ \aChris{ij,k} $ satisfy
\begin{align}
	\aChris{ij,k} (\theta,\gamma) & = \Gamma^g_{ij,k} (\theta,\gamma) - \frac{ \alpha }{ 2 }\Exp[\paren{ \pd_{i}l_{X_{1},\theta,\gamma} }\paren{ \pd_{j}l_{X_{1},\theta,\gamma} }\paren{ \pd_{k}l_{X_{1},\theta,\gamma} }]
	\label{eq:alpha-connection}
\end{align}
for $ i,j,k=\Ilist $.
We will use the Einstein notation for the two types of indices $ a,b,c,d \in \braces{ \Ilist* }$ and $ i,j,k,l,m \in \braces{ \Ilist } $ throughout this paper. A pair of subscript and superscript indices implies summation over those indices. For example, the terms
\begin{align}
	\sum_{b=1}^{d+1} g^{ab}\pd_{b},\quad \sum_{j=1}^{d} g^{ij}\pd_{j}
\end{align}
are abbreviated as $ g^{ab}\pd_{b} $ and $ g^{ij}\pd_{j} $, where $ g^{ij}=g^{ij}(\theta,\gamma) $. The symbol $ g^{ij} $ is the $ (i,j) $ component of $ g^{-1} $.

\begin{example}[Truncated exponential distributions]
	Consider the family of truncated exponential distributions with the density
	\begin{align}
		p(x;\theta,\gamma) & = \theta \e^{-\theta(x-\gamma)}\cdot \indi_{[\gamma,\infty)}(x) \quad \paren{x\in \bR} \label{eq:truncated_exponential_density}
	\end{align}
	with $ \Theta = \paren{ 0,\infty }, I = \bR $ and $ q(x;\theta) = \e^{-\theta x} $.
	It follows that $ \psi(\theta,\gamma) = -\theta \gamma - \log \theta $.
	This family is an oTEF with $ d=1 $.
\end{example}

\begin{example}(Truncated normal distribution)
	Let $ \model $ be the family of truncated normal distributions with the density
	\begin{align}
		p(x;\mu,\sigma,\gamma) & =\frac{1}{\sigma}\phi\paren{\frac{ x-\mu }{ \sigma }}\exp\braces{ -\log(1-\Phi(\frac{ \gamma-\mu }{ \sigma })) } \cdot \indi_{[\gamma,\infty)}(x) \quad \paren{x\in \bR}
		\label{eq:truncated_normal_density}
	\end{align}
	with $ (\mu,\sigma) \in \Theta = \bR \times \paren{ 0,\infty },\, \gamma \in \bR, \, q(x;\mu,\sigma)=\phi((x-\mu)/\sigma) $. Here, $ \phi(x) $ and $ \Phi(x) $ are the density and the distribution function of the standard normal distribution, respectively.
	Hereafter, $ N(\mu,\sigma,\gamma) $ denotes the truncated normal distribution with the above density.
	The family of truncated normal distributions is also an oTEF with the natural parameters $ (\alpha,\beta)=(\mu/\sigma^{2}, -1/2\sigma^{2})\in \bR\times \paren{ -\infty,0 } $ and the density
	\begin{align}
		p(x;\alpha,\beta,\gamma) & = \frac{ 1 }{ \sqrt{ \pi } }\exp\braces{ \alpha x + \beta x^{2} + \frac{ 1 }{ 2 }\log(-\beta)+\frac{ \alpha^{2} }{ 4\beta } -\log\paren{1-\Phi\paren{\nu}} } \cdot \indi_{[\gamma,\infty)}(x)
		\label{eq:truncated_normal_density_natural}
	\end{align}
	for $ x\in \bR $, where $ \nu=\gamma\sqrt{-2\beta}-\frac{\alpha}{\sqrt{-2\beta}} $.
\end{example}

\section{Probability matching priors for non-regular models with regular multiparameters}\label{sec:probability_matching_priors}

Let $ U^{i}\coloneqq \sqrt{n}\paren{\theta^{i}- \hat{\theta}^{i}_{\ML} }\quad \mathrm{for} \quad i=\Ilist,\, U\coloneqq \paren{ U^{1},\ldots,U^{d} }^{\top}$, and $\,T\coloneqq n\hat{c}\paren{ \gamma-\hat{\gamma}_{\ML} } $.
$ U $ converges in distribution to the normal distribution $ N(0,g_{\theta}^{-1}(\theta,\gamma)) $ as $ n\to\infty $.
On the other hand, $ T $ converges in distribution to the exponential distribution $ \mathrm{Exp}(1) $ as $ n\to\infty $.
Consider a smooth prior density $ \pi $ on $ \Theta \times I$, which satisfies the following property matching the frequentist and posterior probabilities:
\begin{align*}
  P_{\theta,\gamma}^{n}\paren{ \frac{U^i}{\sqrt{\hat{g}^{ii}}} \leq z  }=P_{\pi}^{n}\paren{ \frac{U^i}{\sqrt{\hat{g}^{ii}}} \leq z  \mid \sample } + \bigOp{\frac{ 1 }{ n }}
\end{align*}
for all $ z \in \bR$, where $ \sample=(X_{1},\ldots,X_{n}) $. Here, $ P_{\theta,\gamma}^{n} $ denotes the joint distribution of $ \sample $, and $ P_{\pi}^{n}(\cdot \mid \sample) $ is the posterior probability given $ \sample $.
Such a prior is called a \textit{probability matching prior} for the regular parameter $ \theta^{i}\: (i=1,\ldots,d) $ \citep{datta1995PriorsProvidingFrequentist}, and is denoted by $ \pmp[i] $.
If the prior $ \pi $ also satisfies
\begin{align*}
  P_{\theta,\gamma}^{n}\paren{ T\leq z }=P_{\pi}^{n}\paren{ T \leq z \mid \sample } + \bigOp{\frac{ 1 }{ n^{2} }},
\end{align*}
then the prior $ \pmp[\gamma] $ is a probability matching prior for the truncation parameter $ \gamma $ \citep{ghosal1999ProbabilityMatchingPriors}.
The following theorems extend the results of \cite{ghosal1999ProbabilityMatchingPriors} to the multivariate case, restricted to the oTF.
Ghosal considered only the case $ d=1 $. Our results may also hold for Ghosal's non-regular model, but the proof is more involved.

To derive the probability matching prior for the oTF, we consider the following asymptotic expansion of the posterior density.

\begin{lemma}\label{lem:asymptotic_expansion_posterior_density}
Let $ \hat{\theta}_{\ML} $ and $ \hat{\gamma}_{\ML} $ be the MLEs of $ \theta $ and $ \gamma $.
With $ u=\sqrt{n}\paren{ \theta-\qML },\,t=n\hat{c}(\gamma-\cML) $, the posterior density $ \pi(u,t;\sample) $ admits the asymptotic expansion
\begin{align}
\pi\paren{ u,t;\sample } & = \frac{ 1 }{ \paren{ 2\pi }^{d/2}\sqrt{\det \hat{g}_{\theta}}^{-1} }\e^{t - u^{\top}\hat{g}_{\theta} u/2  }\bracket{ 1+\frac{ 1 }{ \sqrt{n} }B_{1}(u,t) + \frac{ 1 }{ n }B_{2}(u,t) + \bigOp{ \frac{ 1 }{ n^{3/2} } } },
\label{eq:asymptotic_expansion_posterior_density}
\end{align}
where $ \hat{\pi} = \pi(\qML, \cML) $, and
\begin{align*}
B_{1} & = \frac{ 1 }{ \hat{\pi} }D_{\theta}^{\top}\hat{\pi}  u + \hat{A}^{(1,1)\top} u \frac{ t }{ \hat{c} } + \frac{ 1 }{ 3! } \hat{A}^{(3,0)\top}  u^{\otimes{ 3 }}, \\
B_{2} & = \frac{ 1 }{ \hat{c}\hat{\pi} }\pd_{\gamma}\hat{\pi}\paren{ t+1 }
+ \frac{ 1 }{ 2\hat{\pi} } D_{\theta}^{\otimes{ 2 }}\hat{\pi} (u^{\otimes{ 2 }}-\mathrm{vec} \,\hat{g}_{\theta}^{-1}) \\
& \quad +\frac{ 1 }{ \hat{c}\hat{\pi}  } D_{\theta}^{\top}\hat{\pi}\otimes \hat{A}^{(1,1)\top}   \paren{u^{\otimes{ 2 }}t +\mathrm{vec} \,\hat{g}_{\theta}^{-1} }
+ \frac{ 1 }{ 3!\hat{\pi} } D_{\theta}^{\top}\hat{\pi}\otimes\hat{A}^{(3,0)\top}   \paren{ u^{\otimes{ 4 }}- 3\mathrm{vec} \paren{\hat{g}_{\theta}^{-1}}^{\otimes{ 2 }} } \\
& \quad + \frac{ 1 }{ 2\hat{c}^{2} }\hat{A}^{(0,2)}\paren{ t^{2}-2 } - \frac{ 1 }{ 2\hat{c}}\hat{A}^{(2,1)\top} \paren{u^{\otimes{ 2 }}t + \mathrm{vec} \,\hat{g}_{\theta}^{-1} }
+\frac{ 1 }{ 4! }\hat{A}^{(4,0)\top}  \paren{ u^{\otimes{ 4 }}-3\mathrm{vec} \paren{\hat{g}_{\theta}^{-1}}^{\otimes{ 2 }} } \\
& \quad + \frac{ 1 }{2 \hat{c}^{2} }\paren{ \hat{A}^{(1,1)\top} }^{\otimes{ 2 }}   \paren{ u^{\otimes{ 2 }}t^{2}-2\mathrm{vec} \,\hat{g}_{\theta}^{-1} }
+ \frac{ 1 }{2\cdot 3!^{2} }\paren{ \hat{A}^{(3,0)\top} }^{\otimes{ 2 }} S_{6}   \paren{ u^{\otimes{ 6 }}-15\mathrm{vec} \paren{\hat{g}_{\theta}^{-1}}^{\otimes{ 3 }}  } \\
& \quad +\frac{ 1 }{ 3!\hat{c} }\paren{ \hat{A}^{(1,1)} \otimes \hat{A}^{(3,0)} }^{\top}  \paren{  u^{\otimes{ 4 }}t+3\mathrm{vec} \paren{\hat{g}_{\theta}^{-1}}^{\otimes{ 2 }} }.
\end{align*}
Here, $ S_{6} \in \bR^{d^{6}\times d^{6}}$ is the symmetrizer matrix defined by \eqref{eq:symmetrizer_matrix_definition} in Appendix \ref{sec:appendix_pf_asymptotic_expansion_posterior_density}.
\end{lemma}

For a matrix $ A=\paren{ a_{ij} }\in \bR^{m_{1}\times m_{2} }\; \paren{m_{1},m_{2}\in \bN  }$, the \texttt{vec} operator is a linear map from $ \bR^{m_{1}\times m_{2}} \to \bR^{m_{1}m_{2}} $ that stacks the columns of $ A $ into a single column vector:
\begin{align}
  \mathrm{vec} A = \paren{ a_{11},\ldots,a_{m_{1}1},\ldots,a_{1m_{2}},\ldots,a_{m_{1}m_{2}} }^{\top}.
\end{align}
The proof of Lemma \ref{lem:asymptotic_expansion_posterior_density} is given in Appendix \ref{sec:appendix_pf_asymptotic_expansion_posterior_density}.

This lemma provides the asymptotic expansion of the posterior probability $ P_{\pi}^{n} $ and the frequentist probability $ P_{\theta,\gamma}^{n} $, using a shrinkage argument.
Then, we derive the conditions for the probability matching prior on an oTF as follows.

\begin{theorem}\label{thm:probability_matching_prior_oTF}
  The probability matching prior $ \pmp[\gamma](\theta,\gamma) $ for the non-regular parameter $ \gamma $ is the solution of the partial differential equation
  \begin{align}
    \pd_{\gamma}\log \pi + A^{(1,1)}_{i} g^{ij} \pd_{j} \log \pi
     & = \pd_{\gamma}\log c - \pd_{i}A_{j}^{(1,1)}g^{ij} \\
     & \quad  + A^{(1,1)}_{i}g^{ij} \braces{ \pd_{j}\log c + \pd_{j}\log \paren{  \det g_{\theta} } - \paren{ \Gamma^{g}_{mk,j} - \Gamma^{g}_{kj,m} }g^{km} }.
    \label{eq:condition_ProbabilityMatchingPrior_gamma_oTF}
  \end{align}
  On the other hand, the probability matching prior $ \pmp[i](\theta,\gamma) $ for the regular parameter $ \theta^{i}\; \paren{ i=\Ilist } $ is the solution of the partial differential equation
  \begin{align}
    \frac{g^{ij}}{\sqrt{g^{ii}}}\pd_{j} \log\pi(\theta,\gamma) =-\pd_{j}\paren{ \frac{g^{ij}}{\sqrt{g^{ii}}} }.
    \label{eq:condition_ProbabilityMatchingPrior_theta_oTF}
  \end{align}
\end{theorem}
The proof is given in Appendix \ref{sec:appendix_pf_probability_matching_prior_oTF}. We can find the solution in the case of the oTEF in Section \ref{sec:Lie_derivative}.

\setcounter{example}{0}
\begin{example}[Truncated exponential distributions (continued)]
  Consider the family of truncated exponential distributions with the density \eqref{eq:truncated_exponential_density}.
  In this case, the Riemannian metric $ g $ is given by
  \begin{align}
    g(\theta,\gamma)
     & = 
    \begin{pmatrix}
      g_{11} & g_{1\gamma} \\
      g_{\gamma 1} & g_{\gamma\gamma}
    \end{pmatrix}, \\
     & = \begin{pmatrix}
      \frac{ 1 }{ \theta^{2} } & 0 \\
      0 & \theta^{2}
    \end{pmatrix}.
  \end{align}
  Note that here $ \theta^{2} $ means the square of $ \theta $, not its second component.
  The $ \alpha $-connection for the regular parameter $ \theta $ is given by
  \begin{align}
    \aChris{111} & = -\frac{1-\alpha }{\theta^3}
  \end{align}
  for $ \alpha \in \bR $. We also have
  \begin{align}
    A^{(1,1)}(\theta,\gamma) & = 1, & c(\theta,\gamma) & = \theta.
  \end{align}

  Thus, the condition for $ \pmp[\gamma](\theta,\gamma) $ from \eqref{eq:condition_ProbabilityMatchingPrior_gamma_oTF} becomes
  \begin{align}
    \pd_{\gamma}\log \pi + \theta^{2} \pd_{\theta}\log \pi = - \theta.
  \end{align}
  The condition for $ \pmp[\theta](\theta,\gamma) $ from \eqref{eq:condition_ProbabilityMatchingPrior_theta_oTF} becomes
  \begin{align}
    \pd_{\theta}\log \pi = -\frac{1}{\theta}.
  \end{align}
\end{example}

\begin{example}[Truncated normal distributions (continued)]
  Consider the family of truncated normal distributions with the density \eqref{eq:truncated_normal_density_natural}, using the natural parameter $ \theta =\paren{ \alpha,\beta }, \gamma $.
  Let $ \Psi(v) \coloneqq \log(1-\Phi(v)) $ and let $ \Psi^{(r)} $ denote the $ r $-th derivative of $ \Psi(v) $ for $ v \in \bR $:
  \begin{align}
    \Psi^{(1)}(v) & = -\frac{ \phi(v) }{ 1-\Phi(v) }, \\
    \Psi^{(2)}(v) & = -\frac{ \phi'(v) }{ 1-\Phi(v) } - \braces{ \frac{ \phi(v) }{ 1-\Phi(v) } }^{2}, \\
    \Psi^{(3)}(v) & = -\frac{ \phi''(v) }{ 1-\Phi(v) } - 3\frac{ \phi'(v)\phi(v) }{\braces{ 1-\Phi(v) }^{2} } - 2\braces{ \frac{ \phi(v) }{ 1-\Phi(v) } }^{3}.
  \end{align}

  The Riemannian metric $ g $ is given by
  \begin{align}
    g_{11} & = -\frac{ 1 }{ 2\beta } + \paren{ \pd_{\alpha}\nu }^{2} \Psi^{(2)}(\nu), \\
    g_{12} & = \frac{ \alpha }{ 2\beta^{2} } + \paren{ \pd_{\alpha}\pd_{\beta}\nu }\Psi^{(1)}(\nu) + \paren{ \pd_{\alpha}\nu }\paren{ \pd_{\beta}\nu }\Psi^{(2)}(\nu), \\
    g_{22} & = \frac{ 1 }{ 2\beta^{2} } - \frac{ \alpha^{2} }{ 2\beta^{3} } + \paren{ \pd_{\beta}\pd_{\beta}\nu }\Psi^{(1)}(\nu) + \paren{ \pd_{\beta}\nu }^{2}\Psi^{(2)}(\nu), \\
    g_{1\gamma} & = g_{2\gamma} = 0, \\
    g_{\gamma\gamma} & = \paren{ \Psi^{(1)}(\nu) }^{2}(\pd_{\gamma}\nu)^{2},
  \end{align}
  where
  \begin{align}
    \pd_{\alpha}\nu & = -\frac{1}{\sqrt{-2\beta}}, &
    \pd_{\beta}\nu & = -\frac{ \gamma }{ \sqrt{-2\beta} } - \frac{ \alpha }{ \sqrt{-2\beta}^{3} }, &
    \pd_{\gamma}\nu & = \sqrt{-2\beta}, \\
    \pd_{\alpha}\pd_{\beta}\nu & = - \frac{ 1 }{ \sqrt{-2\beta}^{3} }, &
    \pd_{\beta}\pd_{\beta}\nu & = -\frac{ \gamma }{ \sqrt{-2\beta}^{3} } - \frac{3 \alpha }{ \sqrt{-2\beta}^{5} }.
  \end{align}

  We also obtain
  \begin{align}
    A^{(1,1)}_{1}(\alpha,\beta,\gamma) & = -\paren{ \pd_{\alpha}\nu }\paren{ \pd_{\gamma}\nu }\Psi^{(2)}(\nu), \\
    A^{(1,1)}_{2}(\alpha,\beta,\gamma) & = -\paren{ \pd_{\beta}\pd_{\gamma} \nu}\Psi^{(1)}(\nu) - \paren{ \pd_{\beta}\nu }\paren{ \pd_{\gamma}\nu }\Psi^{(2)}(\nu), \\
    c(\alpha,\beta,\gamma) & = -\Psi^{(1)}(\nu) \paren{ \pd_{\gamma}\nu }.
  \end{align}

  The conditions for the probability matching priors $ \pmp[\gamma],\pmp[1],\pmp[2] $ in \eqref{eq:condition_ProbabilityMatchingPrior_gamma_oTF} and \eqref{eq:condition_ProbabilityMatchingPrior_theta_oTF} can be computed from the above equations.
\end{example}

\section{Moment matching priors for non-regular models with regular multiparameters}\label{sec:moment_matching_priors}
	This section provides the moment matching priors on a one-sided truncated family.
Before introducing our main results, we briefly review moment matching priors.

Moment matching priors are prior distributions that asymptotically match the Bayesian posterior mean and the maximum likelihood estimator.
\cite{ghosh2011MomentMatchingPriors} proposed this idea for regular statistical models.
In regular models, both the Bayes posterior mean and the MLE are asymptotically normal with order $1/\sqrt{n}$.
A first-order moment matching prior eliminates the $1/n$ discrepancy between these two estimators.

Note that moment matching priors are not generally invariant under parameter transformations \citep{ghosh2011MomentMatchingPriors}.
For example, in exponential families, the moment matching prior for the natural parameters is the Jeffreys prior, while that for the expectation parameters is not.

We derive conditions for moment matching priors in two cases:
(i) when the non-regular parameter $ \gamma $ is of interest;
(ii) when the regular parameter $ \theta^{j} $ is of interest for $ j=1,\ldots,d $.

Let $ \hat{\theta}_{\pi} $ and $ \hat{\gamma}_{\pi} $ denote the posterior means of $ \theta $ and $ \gamma $ under a prior $ \pi $, respectively.
In case (i), the moment matching prior $ \mmp[\gamma] $ is defined as one satisfying
\begin{align}
  \hat{\gamma}_{\pi} - \cML^{*} = \bigOp{n^{-3}},
\end{align}
where $\cML^{*} = \cML - \frac{1}{n\hat{c}}$ is the bias-adjusted MLE~\citep{hashimoto2019MomentMatchingPriors}.
In case (ii), the moment matching prior $ \mmp[j] $ is defined by
\begin{align}
  \hat{\theta}_{\pi}^{j} - \qML^{j} = \bigOp{n^{-3/2}}
\end{align}
for $j = 1, \ldots, d$ \citep{ghosh2011MomentMatchingPriors}.

\cite{hashimoto2019MomentMatchingPriors} derived these priors in a setting with a one-dimensional regular parameter and a non-regular parameter.
Our results extend this to $d$-dimensional regular parameters, restricted to the oTF setting.
The proofs of the following result is given in Appendix \ref{sec:appendix_pf_moment_matching_prior_oTF}.

The asymptotic expansion of the posterior density in Lemma~\ref{lem:asymptotic_expansion_posterior_density} yields the following conditions.

\begin{theorem}\label{thm:moment_matching_prior_oTF}
  The moment matching prior $ \mmp[\gamma](\theta,\gamma) $ for the truncation parameter $ \gamma $ is the solution to
  \begin{align}
    \pd_{\gamma}\log \pi + \frac{1}{2} A^{(2,1)}_{ij}g^{ij} - 2\pd_{\gamma}\log c
    + A^{(1,1)}_{i}g^{ij}\braces{ \pd_{j}\log \pi - 2\pd_{j}\log c + \frac{1}{2}A^{(3,0)}_{jkm}g^{km} } = 0.
    \label{eq:condtion_MomentMatchingPrior_gamma_oTF}
  \end{align}
  The moment matching prior $ \mmp[j](\theta,\gamma) $ for the regular parameter $ \theta^{i} \; (i=1,\ldots,d)$ is the solution to
  \begin{align}
    \pd_{i}\log \braces{ \frac{\pi(\theta,\gamma)}{\pi_{J}(\theta,\gamma)} }
    - \frac{1}{2} \aChris[e]{jk,i}(\theta,\gamma) g^{jk}(\theta,\gamma) = 0,
    \label{eq:condtion_MomentMatchingPrior_theta_oTF}
  \end{align}
  for $ i = 1,\ldots,d $, where $ \pi_{J}(\theta,\gamma) = \sqrt{\det g(\theta,\gamma)} $.
\end{theorem}

The partial differential equation for $ \mmp[i](\theta,\gamma) $ resembles that in regular models $ \Set{ P_{\theta}:\theta \in \Theta } $,
where the moment matching prior satisfies
\begin{align}
  \pd_{i}\log\braces{ \frac{ \pi(\theta) }{ \pi_{J}(\theta) } }
  - \frac{1}{2} \aChris[e]{jk,i}(\theta)g^{jk}(\theta) = 0,
\end{align}
with $ \pi_{J}(\theta) $ being the Jeffreys prior \citep{tanaka2023GeometricPropertiesNoninformative}.

\setcounter{example}{0}

\begin{example}[Truncated exponential distributions (continued)]
  Consider the family of truncated exponential distributions with the density \eqref{eq:truncated_exponential_density}.
  In this case, we have
  \begin{align}
    A^{(2,1)}(\theta,\gamma)       & = 0, &
    A^{(3,0)}(\theta,\gamma)       & = \frac{ 2 }{ \theta^{3} }, \\
    \aChris[e]{11,1}(\theta,\gamma) & = 0, &
    \pi_{J}(\theta,\gamma)         & = 1.
  \end{align}
  The values of $ A^{(1,1)}(\theta,\gamma), c(\theta,\gamma), g(\theta,\gamma) $ are given in Section \ref{sec:probability_matching_priors}.

  Thus, the condition for the moment matching prior $ \mmp[\gamma] $ \eqref{eq:condtion_MomentMatchingPrior_gamma_oTF} is given by
  \begin{align}
    \pd_{\gamma}\log \pi(\theta,\gamma) + \theta^{2} \pd_{\theta} \log \pi(\theta,\gamma) -\theta = 0.
  \end{align}
  The condition for the moment matching prior $ \mmp[\theta] $ \eqref{eq:condtion_MomentMatchingPrior_theta_oTF} is also given by
  \begin{align}
    \pd_{\theta}\log\braces{  \pi(\theta,\gamma)}=0.
  \end{align}
\end{example}

\begin{example}[Truncated normal distributions (continued)]
  Consider the family of truncated normal distributions with the density \eqref{eq:truncated_normal_density_natural} with the natural parameter $ \theta =\paren{ \alpha,\beta },\gamma $.
  The components of $ A^{(3,0)} $ are given by
  \begin{align}
    A^{(3,0)}_{111}(\theta,\gamma) & = - \paren{ \pd_{\alpha} \nu}^{3} \Psi^{(3)} (\nu), \\
    A^{(3,0)}_{112}(\theta,\gamma) & = -\frac{ 1 }{ 2\beta^{2} }- \paren{ \pd_{\alpha} \nu}^{2} \paren{\pd_{\beta}\nu} \Psi^{(3)} (\nu) - 2 \paren{ \pd_{\alpha}\pd_{\beta}\nu }\paren{ \pd_{\alpha}\nu }\Psi^{(2)}(\nu), \\
    A^{(3,0)}_{122}(\theta,\gamma) & = \frac{ \alpha }{ \beta^{3} } - \paren{ \pd_{\alpha} \nu}\paren{ \pd_{\beta}\nu }^{2} \Psi^{(3)} (\nu) - \braces{ 2\paren{ \pd_{\alpha}\pd_{\beta}\nu }\paren{ \pd_{\beta}\nu } + \paren{ \pd_{\beta}\pd_{\beta}\nu }\paren{ \pd_{\alpha}\nu } }\Psi^{(2)}(\nu) \nonumber \\
                                   & \quad - \paren{ \pd_{\alpha}\pd_{\beta}\pd_{\beta}\nu }\Psi^{(1)}(\nu), \\
    A^{(3,0)}_{222}(\theta,\gamma) & = \frac{ 1 }{ \beta^{3} } - \frac{ 3\alpha^{2} }{ 2\beta^{4} } - \paren{ \pd_{\beta}\nu }^{3} \Psi^{(3)} (\nu) - 3\paren{ \pd_{\beta}\pd_{\beta}\nu }\paren{ \pd_{\beta}\nu }\Psi^{(2)}(\nu) - \paren{ \pd_{\beta}\pd_{\beta}\pd_{\beta}\nu }\Psi^{(1)}(\nu).
  \end{align}
  The components of $ A^{(2,1)} $ are given by
  \begin{align}
    A^{(2,1)}_{11}(\theta,\gamma) & = - \paren{ \pd_{\alpha} \nu}^{2} \paren{\pd_{\gamma}\nu} \Psi^{(3)} (\nu), \\
    A^{(2,1)}_{12}(\theta,\gamma) & = - \paren{ \pd_{\alpha}\nu }\paren{ \pd_{\beta}\nu }\paren{ \pd_{\gamma}\nu } \Psi^{(3)}(\nu) - \braces{ \paren{ \pd_{\alpha}\pd_{\beta}\nu } \paren{ \pd_{\gamma}\nu } + \paren{ \pd_{\alpha}\nu } \paren{ \pd_{\beta}\pd_{\gamma}\nu } }\Psi^{(2)}(\nu), \\
    A^{(2,1)}_{22}(\theta,\gamma) & = - \paren{ \pd_{\beta}\nu }^{2}\paren{ \pd_{\gamma}\nu } \Psi^{(3)}(\nu) - \braces{ \paren{ \pd_{\beta}\pd_{\beta}\nu }\paren{ \pd_{\gamma}\nu } + 2 \paren{ \pd_{\beta} \pd_{\gamma} \nu} \paren{ \pd_{\beta}\nu } } \Psi^{(2)} (\nu) \nonumber \\
                                  & \quad - \paren{ \pd_{\beta}\pd_{\beta}\pd_{\gamma}\nu }\Psi^{(1)}(\nu).
  \end{align}
  Note that 
  \begin{align}
    \pd_{\alpha}\pd_{\beta}\pd_{\beta} \nu & = - \frac{ 3 }{ \sqrt{-2\beta}^{5} }, \\
    \pd_{\beta}\pd_{\beta}\pd_{\beta} \nu & = - \frac{ 3\gamma }{ \sqrt{-2\beta}^{5} } - \frac{ 15\alpha }{ \sqrt{-2\beta}^{7} }, \\
    \pd_{\beta}\pd_{\beta}\pd_{\gamma} \nu & = - \frac{ 1 }{ \sqrt{-2\beta}^{3} }.
  \end{align}
  We obtain the condition for the moment matching prior $ \mmp[\gamma] $ in \eqref{eq:condtion_MomentMatchingPrior_gamma_oTF} by the above equations.

  We also obtain the conditions for the moment matching priors $ \mmp[1],\mmp[2] $ in \eqref{eq:condtion_MomentMatchingPrior_theta_oTF} by
  \begin{align}
    \pi_{J}(\theta,\gamma)         & =  \sqrt{g_{11}g_{22}-g_{12}g_{12}}\braces{ -\paren{ \pd_{\gamma}\nu }\Psi^{(1)}(\nu) }, \\
    \aChris[e]{ij,k}(\theta,\gamma) & = 0
  \end{align}
  for $ i,j,k=1,2 $.
\end{example}

\section{The Lie derivative shared by probability and moment matching priors}\label{sec:Lie_derivative}
	
In this section, we derive the relationship between the two matching priors on the oTEF.
On the oTEF $ \model_{e} $, the conditions for matching priors have the simple form.
Then, we obtain a common point of the two matching priors for the truncation parameter $ \gamma $ that the Lie derivative appears in the conditions.

Let $ X_{1},\ldots, X_{n} $ be i.i.d. random samples from a distribution $ p(x;\theta,\gamma) $ in an oTEF $ \model_{e} $.
We use the same assumptions and symbols as in the previous sections without the statistical model.
Since the density function of the oTEF has the form \eqref{eq:oTEF-density}, it follows that
\begin{align}
  D_{\theta}A^{(r,0)}(\theta,\gamma) & = A^{(r+1,0)}(\theta,\gamma) = -D_{\theta}^{\otimes r+1} \psi(\theta,\gamma) \quad \paren{ r=0,1,2,\ldots }. 
\end{align}
This property simplifies the conditions for the matching priors.

Consider the vector field $ \chi \coloneqq \pd_{\gamma} + A_{i}^{(1,1)} g^{ij}\pd_{j }$ on $ \model_{e} $.
Let $ \mathcal{L}_{\chi} $ be a Lie derivative along $ \chi $.
Then, the conditions \eqref{eq:condition_ProbabilityMatchingPrior_gamma_oTF} for the probability matching prior $ \pmp[\gamma] $ on the oTEF
have the form
\begin{align}
  \mathcal{L}_{\chi} \braces{ \log \pi -\log \paren{ \det g_{\theta} } - \frac{ 1 }{ 2 }\log g_{\gamma\gamma}} 
  & = 0. \label{eq:condition_pmp_gamma_oTEF_LieD}
\end{align}
On the other hands, the conditions \eqref{eq:condtion_MomentMatchingPrior_gamma_oTF} for the moment matching priors $ \mmp[\gamma] $ on the oTEF
have the form
\begin{align}
  \mathcal{L}_{\chi} \braces{ \log \pi -\frac{ 1 }{ 2 }\log \paren{ \det g_{\theta} } -  \log g_{\gamma\gamma}} 
  & = 0. \label{eq:condition_mmp_gamma_oTEF_LieD}
\end{align}
Thus, the two matching priors have a common point that the Lie derivative $ \mathcal{L}_{\chi} $ appears in the partial differential equations.

The vector field $ \chi $ is regarded as the natural vector field on the other coordinate.
Consider the reparametrization $ (\theta,\gamma) \mapsto (\eta,\gamma')$ with $ \eta = D_{\theta}\psi(\theta,\gamma) $ and $ \gamma'=\gamma $. 
We call $ \eta $ \textit{expectation parameters} since it follows that
$ \eta_{i}= \Exp[F_{i}(X)]$ for $ i=\Ilist $ on the oTEF.
In this case, the parameter space is the set $ H\coloneqq \Set{(\eta(\theta,\gamma),\gamma):\theta\in \Theta,\gamma\in I} $.
Then, the natural vector fields of $ (\eta,\gamma) $ are written as
\begin{align}
  \frac{ \pd }{ \pd \eta_{i} }  & =  g^{ij}\frac{ \pd }{ \pd \theta^{i} },   &\frac{ \pd }{ \pd \gamma' }    & = \chi.
\end{align}

The Lie derivative $ \mathcal{L}_{\chi} $ gives the characterization of the two types of matching priors on the following submodel.
Let $     H'  \coloneqq \Set{\eta(\theta,\gamma):\theta\in \Theta,\gamma\in I},\; I'_{\eta_{0}}  \coloneqq \Set{\gamma:(\eta_{0},\gamma)\in H}$ for $ \eta_{0}\in H' $.
We call $ \evol{\rho}{\tau} \coloneqq \paren{ \det g_{\theta} }^{(\rho + 1/2)} \paren{ g_{\gamma\gamma} }^{(\tau + 1/2)} $ an \textit{extended volume element} of the oTEF for $ \rho,\tau \in \bR $
with coordinate $ \theta,\gamma $.
\begin{theorem}
  Consider the submodel $ \model_{e,\eta_{0}}\coloneqq \Set{p(x;\eta_{0},\gamma):\gamma\in I'_{\eta_{0}}} $ for a fixed $ \eta_{0}\in H' $.
  Then, there exist a unique probability matching prior $ \pmp[\gamma] $ and a unique moment matching prior $ \mmp[\gamma] $ on $ \model_{e,\eta_{0}} $ such that
  \begin{align}
    \pmp[\gamma]\paren{ \gamma } & \propto \evol{1/2}{0}(\eta_{0},\gamma) = \braces{ \det g_{\theta}(\eta_{0},\gamma) }\sqrt{ g_{\gamma\gamma}(\eta_{0},\gamma) } , \\ 
    \mmp[\gamma]\paren{ \gamma } & \propto  \evol{0}{1/2}(\eta_{0},\gamma) = \sqrt{ \det g_{\theta}(\eta_{0},\gamma) }\;{ g_{\gamma\gamma} (\eta_{0},\gamma)} .
  \end{align}
\end{theorem}

\begin{proof}
  On the model $ \model_{e,\eta_{0}} $, the two conditions \eqref{eq:condition_pmp_gamma_oTEF_LieD} and \eqref{eq:condition_mmp_gamma_oTEF_LieD} are given by
  \begin{align}
    \pd_{\gamma} \braces{ \log \pi(\gamma) -\log \paren{ \evol{1/2}{0}(\eta_{0},\gamma) }} & = 0, \\
    \pd_{\gamma} \braces{ \log \pi(\gamma) -\log \paren{ \evol{0}{1/2}(\eta_{0},\gamma) }} & = 0.
  \end{align}
  Thus, $ \evol{1/2}{0}(\eta_{0},\gamma) $ and $ \evol{0}{1/2}(\eta_{0},\gamma) $ are unique matching priors, respectively.
\end{proof}

Some $ \alpha $-parallel priors on the oTEF are matching priors.
$ \alpha $-parallel priors are the extensions of the Jeffreys prior from the geometric point of view (See \cite{takeuchi2005AlphaparallelPriorIts} for details).
The priors are originally defined for regular models, but they can be extended for the oTEF with the $ \alpha $-connections in \eqref{eq:alpha-connection} \citep{yoshioka2023AlphaparallelPriorsOneSided}.
The explicit form of the $ \alpha $-parallel priors on the oTEF is given by
\begin{align*}
  \pi^{(\alpha)}(\theta,\gamma) & \propto \evol{0}{\alpha/2}(\theta,\gamma) 
\end{align*}
for $ \alpha \in \bR $. Note that $ \pi^{(0)}=\pi_{J} $.
Then, $ \pi^{(0)},\pi^{(-1)} $ and the square of $ \pi^{(1/2)} $ satisfy the conditions \eqref{eq:condtion_MomentMatchingPrior_theta_oTF}, \eqref{eq:condition_pmp_gamma_oTEF_LieD} and \eqref{eq:condition_mmp_gamma_oTEF_LieD}, respectively.
The prior $ \pi^{(0)} $ is a moment matching prior for the natural parameter $ \theta $. Also, $ \pi^{(-1)} $ is a probability matching prior for the truncation parameter $ \gamma $.
The square of $ \pi^{(1/2)} $ is a moment matching prior for the truncation parameter $ \gamma $.

\setcounter{example}{0}

\begin{example}[Truncated exponential distributions (continued)]
  Consider the family of truncated exponential distributions with the density \eqref{eq:truncated_exponential_density}.
  This family is one of the oTEFs.
  In Section \ref{sec:probability_matching_priors} and Section \ref{sec:moment_matching_priors}, we have the conditions for the two matching priors for this family.
  Then, 
  \begin{align}
    \pmp[\gamma](\theta,\gamma) & \propto \frac{ 1 }{ \theta }, & \pmp[\theta](\theta,\gamma) & \propto \frac{1}{\theta},   \\ 
    \mmp[\gamma](\theta,\gamma) & \propto  \theta,   & \mmp[\theta](\theta,\gamma) & \propto 1
  \end{align}
  are one of the solutions of the partial differential equations of the conditions, respectively.

  The vector field $ \chi $ on this family is given by
  \begin{align}
    \chi = \pd_{\gamma} +\theta^{2} \pd_{\theta}.
  \end{align}
  The streamline along $ \chi $ is shown in Figure \ref{fig:StreamPlot_exponential_theta_gamma}.

  \begin{figure}
    \centering
    \includegraphics[width=0.5\textwidth]{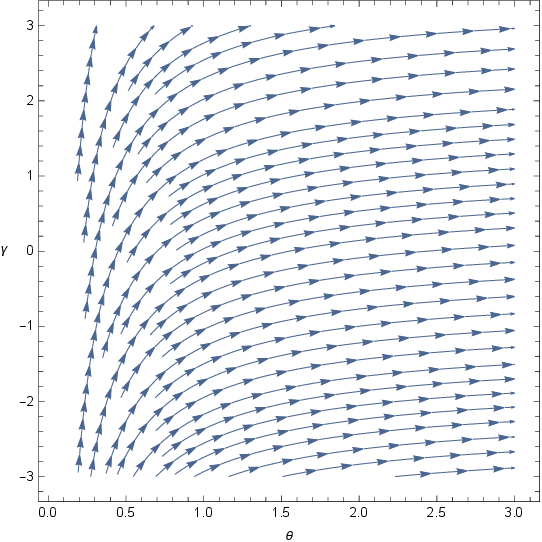}
    \caption{Streamline: $ Exp(\theta,\gamma) $}
    \label{fig:StreamPlot_exponential_theta_gamma}
  \end{figure}

  The expectation parameter is $ \eta=1/\theta+\gamma $ with the parameter space $ H=\Set{(\eta,\gamma):\eta>\gamma} $.
  Then, there exist unique probability and moment matching priors, denoted by $\pmp[\gamma]$ and $\mmp[\gamma]$, respectively, such that
  \begin{align*}
    \pmp[\gamma](\eta_0, \gamma) &\propto \frac{1}{\eta_0 - \gamma}, &
    \mmp[\gamma](\eta_0, \gamma) &\propto \eta_{0} - \gamma
  \end{align*}
  on the submodel $\Set{ p(x; \eta_0, \gamma) : \gamma< \eta_{0} }$ for fixed $\eta_0 \in \mathbb{R}$.
  Note that each streamline in Figure~\ref{fig:StreamPlot_exponential_theta_gamma} represents one such submodel in the $ (\theta,\gamma) $ coordinate.
\end{example}

\begin{example}[Truncated normal distributions (continued)]
  Consider the family of truncated normal distributions with the fixed $ \beta=-1/2 $ in the density \eqref{eq:truncated_normal_density} for simplicity.
  The density $ p(x;\alpha,-1/2,\gamma) $ represents the truncated normal distribution with $ \mu=\alpha $ and $ \sigma = 1$.

  In this case, the vector field $ \chi $ is given by
  \begin{align}
    \chi = \pd_{\gamma} + \frac{ \Psi^{(2)}(\nu) }{1 + \Psi^{(2)}(\nu) }\pd_{\alpha}.
  \end{align}
  The streamline along $ \chi $ shown in Figure \ref{fig:StreamPlot_normal_alpha_gamma}.
  \begin{figure}
    \centering
    \includegraphics[width=0.5\textwidth]{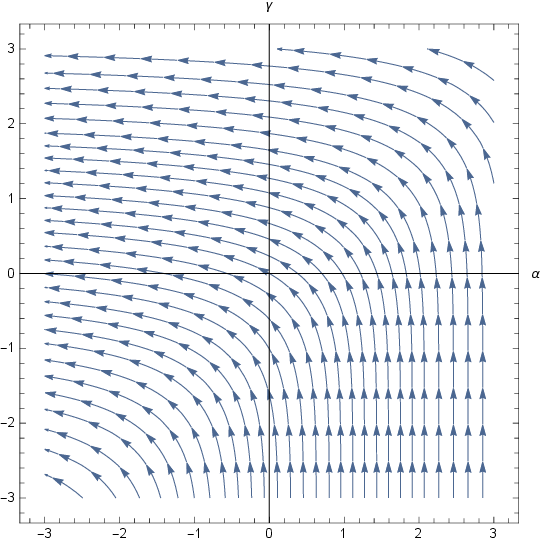}
    \caption{Streamline: $ N(\alpha,1,\gamma) $}
    \label{fig:StreamPlot_normal_alpha_gamma}
  \end{figure}

  The expectation parameter is 
    $ \eta  =  \alpha - \Psi^{(1)}(\nu)  $ with the parameter space $ H=\bR^{2} $.

  The same calculation can be performed for the full family of truncated normal distributions.
\end{example}

\section{Concluding remarks}\label{sec:conclusions}
	This paper reveals the common geometric structure of the probability and moment matching prior in multivariate non-regular models. In particular, we derive the partial differential equations characterizing the two types of matching priors within the framework of an oTF, as presented in Theorems \ref{thm:probability_matching_prior_oTF} and \ref{thm:moment_matching_prior_oTF}. These equations involve the connection coefficients, which do not appear in the univariate case ($d = 1$) \citep{ghosal1999ProbabilityMatchingPriors,hashimoto2019MomentMatchingPriors}.

These differential equations simplify for the subclass of oTF known as oTEF due to the vanishing of the e-connection coefficients when expressed in terms of the natural parameter $\theta$. Under this restriction, we can express the conditions for both matching priors in terms of the Lie derivative along a common vector field. This formulation highlights an invariance of a generalized volume element with respect to differentiation in the direction of the truncation parameter under the parameter transformation.

While the geometric formulation is made explicit in the oTEF case, extending this structure to more general non-regular models remains an open problem. Although the differential equations for both matching priors can be analogously derived, expressing their conditions by geometric terms requires further understanding of the underlying geometric properties of non-regular models.

	\backmatter

	\bmhead{Supplementary information}

	\bmhead{Acknowledgements}

	This work was supported by JSPS KAKENHI grant number JP23K11006 and JST SPRING grant number JPMJSP2138.

\section*{Declarations}

	% Some journals require declarations to be submitted in a standardised format. Please check the Instructions for Authors of the journal to which you are submitting to see if you need to complete this section. If yes, your manuscript must contain the following sections under the heading `Declarations':

	% \begin{itemize}
	%     \item Funding
	%     \item Conflict of interest/Competing interests (check journal-specific guidelines for which heading to use)
	%     \item Ethics approval and consent to participate
	%     \item Consent for publication
	%     \item Data availability
	%     \item Materials availability
	%     \item Code availability
	%     \item Author contribution
	% \end{itemize}

	\bmhead{Competing interests}
	The authors have no competing interests to declare that are relevant to the content of this article.

	\noindent
	% If any of the sections are not relevant to your manuscript, please include the heading and write `Not applicable' for that section.

	\bigskip
	\begin{flushleft}%
	    Editorial Policies for:

	    \bigskip\noindent
	    Springer journals and proceedings: \url{https://www.springer.com/gp/editorial-policies}

	    \bigskip\noindent
	    Nature Portfolio journals: \url{https://www.nature.com/nature-research/editorial-policies}

	    \bigskip\noindent
	    \textit{Scientific Reports}: \url{https://www.nature.com/srep/journal-policies/editorial-policies}

	    \bigskip\noindent
	    BMC journals: \url{https://www.biomedcentral.com/getpublished/editorial-policies}
	\end{flushleft}

	\begin{appendices}




\section{Proof of Lemma \ref{lem:asymptotic_expansion_posterior_density}}
\label{sec:appendix_pf_asymptotic_expansion_posterior_density}

Let
\begin{align*}
  \tilde{\theta}_{u} & \coloneqq \hat{\theta}_{\ML}  +\frac{ u }{ \sqrt{n}}, & \tilde{\gamma}_{t} & \coloneqq \hat{\gamma}_{\ML} +\frac{ t }{ n\hat{c} }.
\end{align*}
The posterior density $ \pi(u,t;\sample) $ is given by
\begin{align}
  \pi(u,t;\sample) & = \frac{  \pi\paren{ \tilde{\theta}_{u},\tilde{\gamma}_{t} }
    \prod_{i=1}^{n}p(X_{i};\tilde{\theta}_{u},\tilde{\gamma}_{t}) }
  {  \int \pi\paren{ \tilde{\theta}_{u'},\tilde{\gamma}_{t'} }
  \prod_{i=1}^{n}p(X_{i};\tilde{\theta}_{u'},\tilde{\gamma}_{t'})du' dt' }        \\ %{ 分子\pi \prod p}{分母 \int \pi \prod p dt' du'}  
                   & = \frac{  \pi\paren{ \tilde{\theta}_{u},\tilde{\gamma}_{t} }
  \exp\bracket{ \sum_{i=1}^{^{n}}\braces{ \log p (X_{i};\tilde{\theta}_{u},\tilde{\gamma}_{t})- \log p (X_{i} ; \hat{\theta}_{\ML},\hat{\gamma}_{\ML} ) } } }
  {  \int \pi\paren{ \tilde{\theta}_{u'},\tilde{\gamma}_{t'} }
  \exp\bracket{ \sum_{i=1}^{^{n}}\braces{ \log p (X_{i};\tilde{\theta}_{u'},\tilde{\gamma}_{t'})- \log p (X_{i};\hat{\theta}_{\ML},\hat{\gamma}_{\ML} ) } }du' dt' }
  \label{eq:posterior_density_oTF}
\end{align}
For the calculation of the asymptotic expansion of the posterior density,
we will calculate the asymptotic expansion of three terms in \eqref{eq:posterior_density_oTF}: $ \pi\paren{ \tilde{\theta}_{u},\tilde{\gamma}_{t} } $, the exponential term, and the denominator.

First, we derive the asymptotic expansion of the prior term $ \pi\paren{ \tilde{\theta}_{u},\tilde{\gamma}_{t} } $.
By Taylor's theorem, we have
\begin{align}
  \pi\paren{ \tilde{\theta}_{u},\tilde{\gamma}_{t} }
   & = \hat{\pi}+D_{\theta}^{\top}\hat{\pi}  \frac{ u }{ \sqrt{n}}+\pd_{\gamma}\hat{\pi}\frac{ t }{ n\hat{c} }                      \\
   & \quad+\frac{ 1 }{ 2 } \paren{   \frac{ u }{ \sqrt{n}} }^{\top} D_{\theta}D_{\theta}^{\top}\hat{\pi}  \frac{ u }{ \sqrt{n}}
    +\bigOp{ \frac{ 1 }{ n^{3/2} } }                                                                                            \\
   & = \hat{\pi} + \frac{ 1 }{ \sqrt{n} }P_{1}(u) + \frac{ 1 }{ n }P_{2}(u,t) + \bigOp{ \frac{ 1 }{ n^{3/2} } },
  \label{eq:asymptotic_expansion_prior_oTF}
\end{align}
where
\begin{align*}
  \hat{\pi}  & \coloneqq  \pi\paren{ \hat{\theta}_{\ML},\hat{\gamma}_{\ML} },                                                           \\
  P_{1}(u)   & \coloneqq D_{\theta}^{\top}\hat{\pi}  u,                                                                                 \\
  P_{2}(u,t) & \coloneqq \pd_{\gamma}\hat{\pi}\frac{ t }{ \hat{c} } + \frac{ 1 }{ 2 } u^{\top} D_{\theta}D_{\theta}^{\top}\hat{\pi}  u.
\end{align*}

Second, we derive the asymptotic expansion of the exponential term in \eqref{eq:posterior_density_oTF}.
Let $  \tilde{l}_{i}(u,t) \coloneqq \log p (X_{i};\tilde{\theta}_{u},\tilde{\gamma}_{t})$ and $ \hat{l}_{i} \coloneqq \log p (X_{i};\hat{\theta}_{\ML},\hat{\gamma}_{\ML} ) $.
By Taylor's theorem, we get the asymptotic expansion of  $ \tilde{l}_{i}(u,t) - \hat{l}_{i} $ as
% $ \log p (X_{i};\tilde{\theta}_{u},\tilde{\gamma}_{t}) - \log p (X_{i};\hat{\theta}_{\ML},\hat{\gamma}_{\ML} ) $ as
\begin{align*}
  \tilde{l}_{i}(u,t)- \hat{l}_{i}
   & = \paren{ D_{\theta}\hat{l}_{i} }^{\top}\paren{ \frac{u}{ \sqrt{n} } } + \pd_{\gamma}\hat{l}_{i} \paren{ \frac{ t }{ n\hat{c} } }                                              \\
   & \quad +\frac{ 1 }{ 2 }\paren{ D_{\theta}^{\otimes{ 2 }}\hat{l}_{i} }^{\top}\paren{ \frac{u}{ \sqrt{n} }  }^{\otimes{ 2 }}  +\frac{ 1 }{ 2 }\pd_{\gamma}\pd_{\gamma}\hat{l}_{i}
  \paren{ \frac{ t }{ n\hat{c} } }^{2}                                                                                                                                              \\
   & \quad + \paren{ D_{\theta}\pd_{\gamma}\hat{l}_{i}  }^{\top}  \paren{ \frac{u}{ \sqrt{n} }  } \paren{ \frac{ t }{ n\hat{c} } }                                                  \\
   & \quad + \frac{ 1 }{ 3! }\paren{ D_{\theta}^{\otimes{ 3 }}\hat{l}_{i} }^{\top}\paren{ \frac{u}{ \sqrt{n} }  }^{\otimes{ 3 }}                                                    \\
   & \quad + \frac{3}{3!} \paren{ D_{\theta}^{\otimes{ 2 }}\pd_{\gamma}\hat{l}_{i} }^{\top}\paren{ \frac{u}{ \sqrt{n} }  }^{\otimes{ 2 }}\paren{ \frac{ t }{ n\hat{c} } }           \\
   & \quad + \frac{1}{4!} \paren{ D_{\theta}^{\otimes{ 4 }}\hat{l}_{i} }^{\top}\paren{ \frac{u}{ \sqrt{n} }  }^{\otimes{ 4 }}   +\bigOp{ \frac{ 1 }{ n^{5/2} } }                                                                                                                                         \\
   & = \frac{ 1 }{ \sqrt{n} }\paren{ D_{\theta}\hat{l}_{i} }^{\top}u + \frac{ 1 }{ n }\pd_{\gamma}\hat{l}_{i} \paren{ \frac{ t }{ \hat{c} } }                                       \\
   & \quad +\frac{ 1 }{ 2n }\paren{ D_{\theta}^{\otimes{ 2 }}\hat{l}_{i} }^{\top}u^{\otimes{ 2 }}
  + \frac{ 1 }{ n^{3/2} }\paren{ D_{\theta}\pd_{\gamma}\hat{l}_{i}  }^{\top}  u \paren{ \frac{ t }{ \hat{c} } }                                                                     \\
   & \quad + \frac{ 1 }{ 3!n^{3/2} }\paren{ D_{\theta}^{\otimes{ 3 }}\hat{l}_{i} }^{\top}u^{\otimes{ 3 }}                                                                           \\
   & \quad + \frac{ 1 }{ 2n^{2} }\pd_{\gamma}\pd_{\gamma}\hat{l}_{i} \paren{ \frac{ t }{ \hat{c} } }^{2}
  + \frac{1}{2n^{2}} \paren{ D_{\theta}^{\otimes{ 2 }}\pd_{\gamma}\hat{l}_{i} }^{\top}u^{\otimes{ 2 }}\paren{ \frac{ t }{ \hat{c} } }                                               \\
   & \quad + \frac{1}{4!n^{2}} \paren{ D_{\theta}^{\otimes{ 4 }}\hat{l}_{i} }^{\top}u^{\otimes 4} +\bigOp{ \frac{ 1 }{ n^{5/2} } }.
\end{align*}
Recall that $ \sum_{i} D_{\theta}\hat{l}_{i}=0$, $\hat{g}_{\theta}=  -\sum_{i} D_{\theta}D_{\theta}^{\top}\hat{l}_{i}/n$,
$  \hat{c}=\sum_{i} \pd_{\gamma}\hat{l}_{i}/n$, and $\hat{A}^{(r,s)} =\sum_{i} D_{\theta}^{\otimes{ r }}\paren{ \pd_{\gamma} }^{s}\hat{l}_{i}/n $.
It follows that
\begin{align*}
  \sum_{i=1}^{n} & \paren{ \log p (X_{i};\tilde{\theta}_{u},\tilde{\gamma}_{t})  - \log p (X_{i};\hat{\theta}_{\ML},\hat{\gamma}_{\ML} ) }                                            \\
                 & =  t-\frac{ 1 }{ 2 }u^{\top}\hat{g}_{\theta} u   + \frac{ 1 }{ \sqrt{n} }  \braces{ \hat{A}^{(1,1)\top} u \frac{ t }{ \hat{c} } + \frac{ 1 }{ 3! } \hat{A}^{(3,0)\top}  u^{\otimes{ 3 }} }                    \\
                 & \quad + \frac{ 1 }{ n } \braces{ \frac{1}{2}\hat{A}^{(0,2)} \frac{t^{2}}{\hat{c}^{2}} + \frac{ 1 }{ 2 } \hat{A}^{(2,1)\top} u^{\otimes{ 2 }} \frac{ t }{ \hat{c} }
    + \frac{ 1 }{ 4! } \hat{A}^{(4,0)\top} u^{\otimes{ 4 }}
  }                  +  \bigOp{ \frac{ 1 }{ n^{3/2} } }                                                                                                                            \\
                 & = t-\frac{ 1 }{ 2 } u^{\top}\hat{g}_{\theta} u + \frac{ 1 }{ \sqrt{n} } L_{1}(u,t) + \frac{ 1 }{ n } L_{2}(u,t) + \bigOp{ \frac{ 1 }{ n^{3/2} } },
\end{align*}
where
\begin{align*}
  L_{1}(u,t) & = \hat{A}^{(1,1)\top} u \frac{ t }{ \hat{c} } + \frac{ 1 }{ 3! } \hat{A}^{(3,0)\top}  u^{\otimes{ 3 }},                             \\
  L_{2}(u,t) & = \frac{1}{2}\hat{A}^{(0,2)} \frac{t^{2}}{\hat{c}^{2}} + \frac{ 1 }{ 2 } \hat{A}^{(2,1)\top} u^{\otimes{ 2 }} \frac{ t }{ \hat{c} }
  + \frac{ 1 }{ 4! } \hat{A}^{(4,0)\top} u^{\otimes{ 4 }}.
\end{align*}
Then, the asymptotic expansion of the exponential term is given by
\begin{align}
  \exp & \bracket{ \sum_{i=1}^{^{n}} \braces{ \log p (X_{i};\tilde{\theta}_{u},\tilde{\gamma}_{t})- \log p (X_{i};\hat{\theta}_{\ML},\hat{\gamma}_{\ML} ) } }                                                                                                                    \\
       & = \exp\braces{ t-\frac{ 1 }{ 2 } u^{\top}\hat{g}_{\theta} u + \frac{ 1 }{ \sqrt{n} } L_{1}(u,t) + \frac{ 1 }{ n } L_{2}(u,t) + \bigOp{ \frac{ 1 }{ n^{3/2} } } }                                                                                                        \\
       & = \exp\braces{ t-\frac{ 1 }{ 2 } u^{\top}\hat{g}_{\theta} u }\bracket{ 1+\frac{ 1 }{ \sqrt{n} } L_{1}(u,t) + \frac{ 1 }{ n } \paren{ L_{2}(u,t) + \frac{1}{2}L_{1}^{2}(u,t) } + \bigOp{ \frac{ 1 }{ n^{3/2} } } }. \label{eq:asymptotic_expansion_exponential_term_oTF}
\end{align}

Third, we derive the asymptotic expansion of the denominator of \eqref{eq:posterior_density_oTF}.
From \eqref{eq:asymptotic_expansion_prior_oTF} and \eqref{eq:asymptotic_expansion_exponential_term_oTF}, the numerator of \eqref{eq:posterior_density_oTF} is represented as
\begin{align}
   & \pi\paren{ \tilde{\theta}_{u},\tilde{\gamma}_{t} }
  \exp\bracket{ \sum_{i=1}^{n}\braces{ \log p (X_{i};\tilde{\theta}_{u},\tilde{\gamma}_{t})- \log p (X_{i};\hat{\theta}_{\ML},\hat{\gamma}_{\ML} ) } }                                                                                    \\
   & \qquad= \left[\hat{\pi} + \frac{ 1 }{ \sqrt{n} }P_{1}(u) + \frac{ 1 }{ n }P_{2}(u,t) + \bigOp{ \frac{ 1 }{ n^{3/2} } }\right]                                                                                                        \\
   & \qquad\quad\cdot \exp\braces{ t-\frac{ 1 }{ 2 } u^{\top}\hat{g}_{\theta} u }\bracket{ 1+\frac{ 1 }{ \sqrt{n} } L_{1}(u,t) + \frac{ 1 }{ n } \paren{ L_{2}(u,t) +\frac{ 1 }{ 2 } L_{1}^{2}(u,t) } + \bigOp{ \frac{ 1 }{ n^{3/2} } } } \\
   & \qquad= \exp\braces{ t-\frac{ 1 }{ 2 } u^{\top}\hat{g}_{\theta} u }\left[ \hat{\pi}+\frac{ 1 }{ \sqrt{n} }\braces{ P_{1}(u) + \hat{\pi}L_{1}(u,t) } \right.                                                                          \\
   & \qquad\quad \left.+ \frac{ 1 }{ n }\braces{ P_{1}(u) L_{1}(u,t) + P_{2}(u,t) + \hat{\pi}L_{2}(u,t) + \frac{1}{2}\hat{\pi}L_{1}^{2}(u,t) } +\bigOp{ \frac{ 1 }{ n^{3/2} } }\right] \label{eq:numerator_posterior_density_oTF}
\end{align}
For the description of the integrating of the numerator in \eqref{eq:posterior_density_oTF}, we introduce symmetrizer matrices $ \symope{r} $ as follows (See \cite{holmquist1988MomentsCumulantsMultivariate} for the details).
Let $ e_{1},\ldots,e_{d} $ be the standard basis of $ \bR^{d} $. Symmetrizer matrix $ S_{r}\in \bR^{d^{r}\times d^{r}} $ acts on $ r $-tensor vectors as
\begin{align}
  S_{r} \paren{ e_{i_{1}}\otimes\cdots\otimes e_{i_{r}} } = \frac{1}{r!}\sum_{\pi\in \mathfrak{S}_{r}} e_{i_{\pi(1)}}\otimes\cdots\otimes e_{i_{\pi(r)}}
  \label{eq:symmetrizer_matrix_definition}
\end{align}
for $ i_{1},\ldots,i_{r}=\Ilist $, where $ \mathfrak{S}_{r} $ is the symmetric group of degree $ r $. When $ r=2,3 $, it holds that
\begin{align}
  S_{2}\paren{ e_{i}\otimes e_{j} }              & = \frac{1}{2}\paren{ e_{i}\otimes e_{j} + e_{j} \otimes e_{i} },                                                             \\
  S_{3}\paren{ e_{i}\otimes e_{j}\otimes e_{k} } & = \frac{1}{6}\left(e_{i}\otimes e_{j}\otimes e_{k}+e_{j}\otimes e_{k}\otimes e_{i} + e_{k}\otimes e_{i}\otimes e_{j} \right. \\
                                                 & \left. \quad + e_{i}\otimes e_{k}\otimes e_{j}+ e_{k}\otimes e_{j}\otimes e_{i} + e_{j}\otimes e_{i}\otimes e_{k}  \right)
\end{align}
for $ i,j,k=\Ilist $. Let us go back to the proof.
By integrating the numerator term \eqref{eq:numerator_posterior_density_oTF}  over $ u $ and $ t $, we obtain the denominator in \eqref{eq:posterior_density_oTF}.
The following equations are useful for the calculation of the integrals:
\begin{align*}
  \frac{ 1 }{ \paren{2 \pi}^{d/2} \sqrt{\det \hat{g}_{\theta}^{-1}} } \int_{\bR} u^{\otimes{ r }} \e ^{- u^{\top}\hat{g}_{\theta} u/2}du & =
  \begin{cases}
    0                                                                               & (\text{when $ r $ is odd})   \\
    (r-1)!! \symope{r} \mathrm{vec} \paren{ \hat{g}_{\theta}^{-1}}^{\otimes{ r/2 }} & (\text{when $ r $ is even}),
  \end{cases}                                       \\
  \int_{-\infty}^{0}t^{r}\e ^{t} dt                                                                                                      & =  r!\paren{ -1 }^{r}.
\end{align*}
See \cite{holmquist1988MomentsCumulantsMultivariate} for the first equation.
We calculate the integrals by terms.
Let $ \nu(\hat{g}_{\theta}^{-1})= \paren{2 \pi}^{d/2} \sqrt{\det \hat{g}_{\theta}^{-1}}$.
It holds that
\begin{align}
  \int \braces{ P_{1}(u) + \hat{\pi}L_{1}(u,t) } \e^{t- u^{\top}\hat{g}_{\theta} u/2} dudt=0
\end{align}
since $ \int u \e^{- u^{\top}\hat{g}_{\theta} u/2}du=0 $.
In the same way, it follows that
\begin{align*}
   \int P_{2}(u,t)\e^{t- u^{\top}\hat{g}_{\theta} u/2}dudt   &= \int \braces{ \pd_{\gamma}\hat{\pi}\frac{ t }{ \hat{c} } +
  \frac{ 1 }{ 2 } D_{\theta}^{\otimes{ 2 }\top} \hat{\pi}  u^{\otimes{ 2 }}} \e^{t- u^{\top}\hat{g}_{\theta} u/2}dudt                                                                              \\
   & = \nu(\hat{g}_{\theta}^{-1})\braces{  -\frac{ 1 }{ \hat{c} }\pd_{\gamma}\hat{\pi} + \frac{ 1 }{ 2 } \paren{ D_{\theta}^{\otimes{ 2 } \top}\hat{\pi} }\,\mathrm{vec}\,\hat{g}_{\theta}^{-1} },
\end{align*}
\begin{align*}
   & \int P_{1}(u) L_{1}(u,t) \e^{t- u^{\top}\hat{g}_{\theta} u/2}dudt                                                                                                           \\
   & =  \int\paren{ D_{\theta}^{\top} \hat{\pi}  u }\paren{ \frac{ 1 }{ \hat{c} }\hat{A}^{(1,1)\top}\hat{g}^{-1/2}ut + \frac{ 1 }{ 3! } \hat{A}^{(3,0)\top}
  u^{\otimes{ 3 }} } \e^{t- u^{\top}\hat{g}_{\theta} u/2}dudt                                                                                                                    \\
   & =\frac{ 1 }{ \hat{c} } \paren{ D_{\theta}^{\top}\hat{\pi}\otimes \hat{A}^{(1,1)\top}  }  \int u^{\otimes{ 2 }}t \e^{t- u^{\top}\hat{g}_{\theta} u/2} du dt                  \\
   & \quad + \frac{ 1 }{ 3! } \paren{ D_{\theta}^{\top}\hat{\pi}\otimes\hat{A}^{(3,0)\top}    }\int u^{\otimes{ 4 }} \e^{t- u^{\top}\hat{g}_{\theta} u/2} dudt                   \\
   & = \nu(\hat{g}_{\theta}^{-1})\left\{-\frac{ 1 }{ \hat{c} }\paren{  D_{\theta}^{\top}\hat{\pi}\otimes \hat{A}^{(1,1)\top}  }  \mathrm{vec}\,\hat{g}_{\theta}^{-1}  + \frac{3}{3!}\paren{  D_{\theta}^{\top}\hat{\pi}\otimes\hat{A}^{(3,0)\top}    } \mathrm{vec}\paren{ \hat{g}_{\theta}^{-1} }^{\otimes{ 2 }}\right\}, \\
\end{align*}
\begin{align*}
   & \int L_{2}(u,t)\e^{t- u^{\top}\hat{g}_{\theta} u/2}dudt                                                                                                       \\
   & = \int \braces{ \frac{1}{2}\hat{A}^{(0,2)} \frac{t^{2}}{\hat{c}^{2}} + \frac{ 1 }{ 2 } \hat{A}^{(2,1)\top} u^{\otimes{ 2 }} \frac{ t }{ \hat{c} }
    +\frac{ 1 }{ 4! } \hat{A}^{(4,0)\top} u^{\otimes{ 4 }}
  } \e^{t- u^{\top}\hat{g}_{\theta} u/2}dudt                                                                                                                       \\
   & = \nu(\hat{g}_{\theta}^{-1})\braces{ \frac{ 1 }{ \hat{c}^{2} } \hat{A}^{(0,2)} - \frac{ 1 }{ 2\hat{c}}\hat{A}^{(2,1)\top} \mathrm{vec}\,\hat{g}_{\theta}^{-1}
    +\frac{ 3 }{ 4! }\hat{A}^{(4,0)\top} \mathrm{vec}\,\paren{ \hat{g}_{\theta}^{-1} }^{\otimes{ 2 }} }   ,
\end{align*}
\begin{align*}
   & \int L_{1}^{2}(u,t)\e^{t- u^{\top}\hat{g}_{\theta} u/2}dudt                                                                                                                 \\
   & = \int \paren{ \frac{ 1 }{ \hat{c} }\hat{A}^{(1,1)\top}\hat{g}^{-1/2}ut + \frac{ 1 }{ 3! } \hat{A}^{(3,0)\top}
  u^{\otimes{ 3 }} }^{2} \e^{t- u^{\top}\hat{g}_{\theta} u/2}dudt                                                                                                                \\
   & = \int \frac{ 1 }{ \hat{c}^{2} }\paren{ \hat{A}^{(1,1)\top} }^{\otimes{ 2 }}  u^{\otimes{ 2 }}t^{2} \e^{t- u^{\top}\hat{g}_{\theta} u/2}dudt + \int \frac{ 1 }{ \paren{ 3! }^{2} } \paren{ \hat{A}^{(3,0)\top} }^{\otimes{ 2 }}   u^{\otimes{ 6 }} \e^{t- u^{\top}\hat{g}_{\theta} u/2}dudt                        \\
   & \quad + \int \frac{ 2 }{ 3!\hat{c} }\paren{ \hat{A}^{(1,1)} \otimes \hat{A}^{(3,0)} }^{\top}   u^{\otimes{ 4 }}t \e^{t- u^{\top}\hat{g}_{\theta} u/2}dudt                   \\
   & = \nu(\hat{g}_{\theta}^{-1})\left\{\frac{ 2 }{ \hat{c}^{2} }\paren{ \hat{A}^{(1,1)\top} }^{\otimes{ 2 }}   \mathrm{vec} \,\hat{g}_{\theta}^{-1}
  + \frac{ 15 }{ 3!^{2} }\paren{ \hat{A}^{(3,0)\top} }^{\otimes{ 2 }}  \symope{6 } \mathrm{vec} \paren{\hat{g}_{\theta}^{-1}}^{\otimes{ 3 }} \right.                             \\
   & \quad \left.- \frac{ 1 }{ \hat{c} }\paren{ \hat{A}^{(1,1)} \otimes \hat{A}^{(3,0)} }^{\top}   \mathrm{vec} \paren{\hat{g}_{\theta}^{-1}}^{\otimes{ 2 }}\right\}.
\end{align*}
We omitted the symmetrizers $ S_{2},S_{4} $ in the above equations because of the symmetry of the terms $ \hat{A} $ and $ g_{\theta} $.
Then, the asymptotic expansion of the denominator of \eqref{eq:posterior_density_oTF} is
\begin{align}
  \paren{2\pi}^{d/2} \sqrt{\det \hat{g}_{\theta}^{-1}}\braces{ \hat{\pi}+\frac{ 1 }{ n }K_{n} + \bigO{\frac{ 1 }{ n^{3/2} }} } \label{eq:asymptotic_expansion_denominator_oTF},
\end{align}
where
\begin{align}
  K_{n} & \coloneqq  -\frac{ 1 }{ \hat{c} }\pd_{\gamma}\hat{\pi} + \frac{ 1 }{ 2 } D_{\theta}^{\otimes{ 2 }}\hat{\pi}\, \mathrm{vec} \,\hat{g}_{\theta}^{-1}                      \\
        & \quad -\frac{ 1 }{ \hat{c} }\paren{  D_{\theta}\hat{\pi}\otimes \hat{A}^{(1,1)} }^{\top}   \mathrm{vec} \,\hat{g}_{\theta}^{-1}
  + \frac{3}{3!} \paren{ D_{\theta}\hat{\pi}\otimes\hat{A}^{(3,0)} }^{\top}  \mathrm{vec} \paren{\hat{g}_{\theta}^{-1}}^{\otimes{ 2 }}                                            \\
        & \quad +\hat{\pi}\left\lbrace \frac{ 1 }{ \hat{c}^{2} } \hat{A}^{(0,2)} - \frac{ 1 }{ 2\hat{c}} \hat{A}^{(2,1)\top} \mathrm{vec} \,\hat{g}_{\theta}^{-1}
  +\frac{ 3 }{ 4! } \hat{A}^{(4,0)\top}\mathrm{vec}\,\paren{ \hat{g}_{\theta}^{-1} }^{\otimes{ 2 }}   \right.                                                                     \\
        & \quad + \frac{ 1 }{ \hat{c}^{2} } \paren{ \hat{A}^{(1,1)\top} }^{\otimes{ 2 }}   \mathrm{vec} \,\hat{g}_{\theta}^{-1}
  + \frac{ 15 }{ 2\cdot 3!^{2} } \paren{ \hat{A}^{(3,0)\top} }^{\otimes{ 2 }}  S_{6} \mathrm{vec} \paren{\hat{g}_{\theta}^{-1}}^{\otimes{ 3 }}                                          \\
        & \quad \left. - \frac{ 1 }{ 2\hat{c} } \paren{ \hat{A}^{(1,1)} \otimes \hat{A}^{(3,0)} }^{\top}  \mathrm{vec} \paren{\hat{g}_{\theta}^{-1}}^{\otimes{ 2 }}\right\rbrace.
\end{align}

Finally, we derive the asymptotic expansion of the posterior \eqref{eq:posterior_density_oTF} from the above calculations.

By \eqref{eq:asymptotic_expansion_prior_oTF}, \eqref{eq:asymptotic_expansion_exponential_term_oTF} and \eqref{eq:asymptotic_expansion_denominator_oTF},
it holds that
\begin{align}
  \pi\paren{ u,t;\sample }=\frac{ 1 }{ \paren{ 2\pi }^{d/2}\sqrt{\det \hat{g}_{\theta}^{-1}} }\exp\braces{t - u^{\top}\hat{g}_{\theta} u/2  }\bracket{ 1+\frac{ 1 }{ \sqrt{n} }B_{1}(u,t) + \frac{ 1 }{ n }B_{2}(u,t) + \bigOp{ \frac{ 1 }{ n^{3/2} } } },
  \label{eq:asymptotic_expansion_posterior_density_oTF_appendix}
\end{align}
where
\begin{align}
  B_{1} & = \frac{ P_{1}(u) }{ \hat{\pi} } + L_{1}(u,t)          \\
        & = \frac{ 1 }{ \hat{\pi} }D_{\theta}^{\top}\hat{\pi}  u
  + \hat{A}^{(1,1)\top} u \frac{ t }{ \hat{c} } + \frac{ 1 }{ 3! } \hat{A}^{(3,0)\top}  u^{\otimes{ 3 }},
\end{align}
\begin{align}
  B_{2} & = \frac{ 1 }{ \hat{\pi} } \braces{P_{1}(u) L_{1}(u,t) + P_{2}(u,t) + \hat{\pi}L_{2}(u,t) + \frac{ 1 }{ 2 } \hat{\pi}L_{1}^{2}(u,t) - K_{n}}                                      \\
        & = \frac{ 1 }{ \hat{c}\hat{\pi} }\pd_{\gamma}\hat{\pi}\paren{ t+1 }
  + \frac{ 1 }{ 2\hat{\pi} } \paren{ D_{\theta}^{\otimes{ 2 }}\hat{\pi} }^{\top} (u^{\otimes{ 2 }}-\mathrm{vec} \,\hat{g}_{\theta}^{-1})                                                   \\
        & \quad +\frac{ 1 }{ \hat{c}\hat{\pi}  }\paren{  D_{\theta}\hat{\pi}\otimes \hat{A}^{(1,1)}  }^{\top}  \paren{u^{\otimes{ 2 }}t +\mathrm{vec} \,\hat{g}_{\theta}^{-1} }
  + \frac{ 1 }{ 3!\hat{\pi} } \paren{ D_{\theta}\hat{\pi}\otimes\hat{A}^{(3,0)} }^{\top}  \paren{ u^{\otimes{ 4 }}- 3\mathrm{vec} \paren{\hat{g}_{\theta}^{-1}}^{\otimes{ 2 }} }           \\
        & \quad + \frac{ 1 }{ 2\hat{c}^{2} }\hat{A}^{(0,2)}\paren{ t^{2}-2 } - \frac{ 1 }{ 2\hat{c}}\hat{A}^{(2,1)\top} \paren{u^{\otimes{ 2 }}t + \mathrm{vec} \,\hat{g}_{\theta}^{-1} }
  +\frac{ 1 }{ 4! }\hat{A}^{(4,0)\top} \paren{ u^{\otimes{ 4 }}-3\mathrm{vec} \paren{\hat{g}_{\theta}^{-1}}^{\otimes{ 2 }} }                                                               \\
        & \quad + \frac{ 1 }{2 \hat{c}^{2} }\paren{ \hat{A}^{(1,1)\top} }^{\otimes{ 2 }}   \paren{ u^{\otimes{ 2 }}t^{2}-2\mathrm{vec} \,\hat{g}_{\theta}^{-1} }
  + \frac{ 1 }{2\cdot 3!^{2} }\paren{ \hat{A}^{(3,0)\top} }^{\otimes{ 2 }} S_{6} \paren{ u^{\otimes{ 6 }}-15\mathrm{vec} \paren{\hat{g}_{\theta}^{-1}}^{\otimes{ 3 }}  }                        \\
        & \quad +\frac{ 1 }{ 3!\hat{c} }\paren{ \hat{A}^{(1,1)} \otimes \hat{A}^{(3,0)} }^{\top}  \paren{  u^{\otimes{ 4 }}t+3\mathrm{vec} \paren{\hat{g}_{\theta}^{-1}}^{\otimes{ 2 }} }.
\end{align}

Note that the next term $ B_{3} $ is the polynomial of odd degree with respect to $ u $.
In the Taylor expansions of $ \pi\paren{ \tilde{\theta}_{u},\tilde{\gamma}_{t} } $ and $ \tilde{l}_{i}(u,t) \;(i=\Ilist) $,
the polynomials consisting of only $ u t^{2},\; u^{\otimes 3} t $ and $u^{\otimes 5}$ appear in the terms of order $ n^{-r/2}$
since $  \tilde{\theta}_{u}  = \hat{\theta}_{\ML}  + u/ \sqrt{n},\; \tilde{\gamma}_{t}  = \hat{\gamma}_{\ML} + t/\paren{ n\hat{c} } $.
Then, in the last expansion \eqref{eq:asymptotic_expansion_posterior_density_oTF_appendix}, the term of order $ n^{-3/2} $ is the polynomial of odd degree with respect to $ u $.

Thus, we complete the proof.

\section{Proof of Theorem \ref{thm:probability_matching_prior_oTF}}
\label{sec:appendix_pf_probability_matching_prior_oTF}

\subsection{Probability matching prior for the truncation parameter}

  For proving the theorem, we calculate the asymptotic expansion of the posterior probability $ P_\pi (T\leq z\mid\sample) $
  and the frequentist probability $ P_{\theta,\gamma}^{n}(T\leq z) $.

  By integrating the expansion of the posterior density \eqref{eq:posterior_density_oTF} over $ u $, we have
  \begin{align}
    \pi\paren{ t;\sample }
     & = \e^{t} \left[1+\frac{ 1 }{ n\hat{c} }\left\lbrace
    \pd_{\gamma}\log \hat{\pi} + D_{\theta}^{\top}\log \hat{\pi}\otimes \hat{A}^{(1,1)\top}  \mathrm{vec} \,\hat{g}_{\theta}^{-1} \right. \right. \\
     & \qquad+ \left. \frac{ 1 }{ 2 }\hat{A}^{(2,1)\top}  \mathrm{vec} \,\hat{g}_{\theta}^{-1}
      +\frac{1}{2} \paren{ \hat{A}^{(1,1)} \otimes \hat{A}^{(3,0)} }^{\top}   \mathrm{vec} \paren{\hat{g}_{\theta}^{-1}}^{\otimes{ 2 }}
    \right\rbrace \paren{ t+1 }                                                                                                                   \\
     & \quad \left. + \frac{1}{n\hat{c}^{2}} \braces{
        \frac{1}{2}\hat{A}^{ (0,2) }
        + \frac{1}{2}\paren{ A^{(1,1)\top} }^{\otimes{ 2 }} \mathrm{vec} \,\hat{g}_{\theta}^{-1}}
      (t^{2}-2)+\bigOp{\frac{1}{n^{2}}} \right] .
    \label{eq:posterior_distribution_marginal_t}
  \end{align}
  The term of order $ n^{-3/2} $ vanishes since it is the polynomial of odd degree with respect to $ u $.
  We set $ \alpha = 1 - \e^{z} $
  The posterior probability is
  \begin{align}
    P_{\pi}^{n}(T\leq z) & = \int_{-\infty}^{z}\pi (t;\sample)dt                                                                                            \\
                         & = (1-\alpha)\left[1+\frac{ 1 }{ n\hat{c} }\left\lbrace
    \pd_{\gamma}\log \hat{\pi} + D_{\theta}^{\top}\log \hat{\pi}\otimes \hat{A}^{(1,1)\top}  \mathrm{vec} \,\hat{g}_{\theta}^{-1} \right. \right.           \\
                         & \qquad+ \left. \frac{ 1 }{ 2 }\hat{A}^{(2,1)\top}  \mathrm{vec} \,\hat{g}_{\theta}^{-1}
      +\frac{1}{2} \paren{ \hat{A}^{(1,1)} \otimes \hat{A}^{(3,0)} }^{\top}   \mathrm{vec} \paren{\hat{g}_{\theta}^{-1}}^{\otimes{ 2 }}
    \right\rbrace z                                                                                                                                         \\
                         & \quad \left. + \frac{1}{n\hat{c}^{2}} \braces{
        \frac{1}{2}\hat{A}^{ (0,2) }
        + \frac{1}{2}\paren{ A^{(1,1)\top} }^{\otimes{ 2 }} \mathrm{vec} \,\hat{g}_{\theta}^{-1}}
    z\paren{ z-2 }\right]+\bigOp{\frac{1}{n^{3/2}}}                                                                                                         \\
                         & = (1-\alpha)\left[1+\frac{ 1 }{ nc }\left\lbrace
    \pd_{\gamma}\log \pi + D_{\theta}^{\top}\log \pi\otimes A^{(1,1)\top}  \mathrm{vec} \,g_{\theta}^{-1} \right. \right.                                   \\
                         & \qquad+ \left. \frac{ 1 }{ 2 }A^{(2,1)\top}  \mathrm{vec} \,g_{\theta}^{-1}
    +\frac{1}{2} \paren{ A^{(1,1)} \otimes A^{(3,0)} }^{\top}   \mathrm{vec} \paren{g_{\theta}^{-1}}^{\otimes{ 2 }}
    \right\rbrace z                                                                                                                                         \\
                         & \quad \left. + \frac{1}{nc^{2}} \braces{
    \frac{1}{2}A^{ (0,2) }
    + \frac{1}{2}\paren{ A^{(1,1)\top} }^{\otimes{ 2 }} \mathrm{vec} \,g_{\theta}^{-1}}
    z(z-2)\right]+\bigOp{\frac{1}{n^{3/2}}}                                                                                                                 \\
                         & = (1-\alpha)\left \lbrack 1+\frac{1}{nc}\braces{ \pd_{\gamma}\log \pi +A^{(1,1)\top}g_{\theta}^{-1}D_{\theta}\log \pi }z \right. \\
                         & \quad + \left.\frac{ 1 }{ cn }\braces{ Q_{1}(\theta,\gamma)z+Q_{2}(\theta,\gamma)z(z-2) }
    \right \rbrack +\bigO{\frac{ 1 }{ n^{3/2} }}, \label{eq:asymptotic_expansion_posterior_t}
  \end{align}
  where
  \begin{align}
    Q_{1}(\theta,\gamma) & = \frac{ 1 }{ 2 }A^{(2,1)\top}  \mathrm{vec} \,g_{\theta}^{-1}
    +\frac{1}{2} \paren{ A^{(1,1)} \otimes A^{(3,0)} }^{\top}   \mathrm{vec} \paren{g_{\theta}^{-1}}^{\otimes{ 2 }}, \\
    Q_{2}(\theta,\gamma) & = \frac{1}{2c}A^{ (0,2) }
    + \frac{1}{2c}\paren{ A^{(1,1)\top} }^{\otimes{ 2 }} \mathrm{vec} \,g_{\theta}^{-1}.
  \end{align}
  The shrinkage argument \citep[Section 1.2]{datta2004ProbabilityMatchingPriors}  provides the frequentist probability $ P_{\theta,\gamma}^{n}(T\leq z) $.
  After replacing $ \pi $ by $ \pi_{\delta} $, a density convergence weakly to the measure degenerate at the point $ \theta $, we integrate the expansion of the posterior density \eqref{eq:asymptotic_expansion_posterior_t} with respect to $ \pi_{\delta} $ and letting $ \delta\to 0 $.
  Here, it follows that
  \begin{align}
    \lim_{\delta \downarrow 0} \int \frac{1}{c(x,y)} \pd_{\gamma}\log \paren{ \pi_{\delta} (x,y) } \pi_{\delta}(x,y) dx dy                                   & = -\pd_{\gamma}\paren{ \frac{ 1 }{ c(\theta,\gamma) } },                                                 \\
    \lim_{\delta \downarrow 0} \int \frac{1}{c(x,y)}A^{(1,1)\top}(x,y)g_{\theta}^{-1}(x,y)D_{\theta}\log \paren{ \pi_{\delta}(x,y) } \pi_{\delta}(x,y) dx dy & = - D_{\theta}^{\top} \paren{ \frac{ 1 }{ c }  g_{\theta}^{-1}(\theta,\gamma)A^{(1,1)}(\theta,\gamma) },
  \end{align}
  Then, we obtain the expansion of the probability $ P_{\theta,\gamma}^{n} (T\leq z) $ as follows:
  \begin{align}
    P_{\theta,\gamma}^{n}(T \leq z) & = (1-\alpha)\left \lbrack 1 + \frac{1}{n}\braces{ -\pd_{\gamma}\paren{ \frac{ 1 }{ c } } - D_{\theta}^{\top} \paren{ \frac{ 1 }{ c }  g_{\theta}^{-1}A^{(1,1)} } }z \right. \\
                                    & \qquad +\left.\frac{ 1 }{ cn }\braces{ Q_{1}(\theta,\gamma)z+Q_{2}(\theta,\gamma)z(z-2) }
    \right \rbrack +\bigO{\frac{ 1 }{ n^{3/2} }}\label{eq:asymptotic_expansion_frequentist_probability_t}
  \end{align}
  Then, by comparing \eqref{eq:asymptotic_expansion_posterior_t} and \eqref{eq:asymptotic_expansion_frequentist_probability_t},
  we get the conditions for the probability matching prior $ \pmp[\gamma] $ that a prior $ \pi $ satisfies the partial differential equation
  \begin{align}
    \pd_{\gamma}\log \pi +A^{(1,1)\top}g_{\theta}^{-1}D_{\theta}\log \pi & = c\braces{  -\pd_{\gamma}\paren{ \frac{ 1 }{ c } } - D_{\theta}^{\top} \paren{ \frac{ 1 }{ c }  g_{\theta}^{-1}A^{(1,1)} } }
    \label{eq:condition_ProbabilityMatchingPrior_gamma_oTF_tensorial}
  \end{align}
  Here, it follows that
  \begin{align}
    D_{\theta}^{\top} \paren{ \frac{ 1 }{ c }  g_{\theta}^{-1}A^{(1,1)} } = -\frac{1}{c}A^{(1,1)\top}g_{\theta}^{-1}D_{\theta}\log c+ \frac{1}{c}\paren{ D_{\theta}^{\top}g_{\theta}^{-1} }  A^{(1,1)}
    +\frac{1}{c} \paren{ D_{\theta}A^{(1,1)} }^{\top} \mathrm{vec}g_{\theta}^{-1}
  \end{align}
  and the components of $ D_{\theta}^{\top}g_{\theta}^{-1}A^{(1,1)} $ are written as
  \begin{align}
    \paren{ \pd_{i}g^{ij} }A^{(1,1)}_{j} & =  -\paren{ \pd_{i}g_{km}}g^{ik}g^{jm} A^{(1,1)}_{j}                                                                                    \\
                                         & = -\paren{ \pd_{m}g_{ik} + \Gamma^{g}_{ik,m} - \Gamma^{g}_{km,i} }g^{ik}g^{jm} A^{(1,1)}_{j}                                            \\
                                         & = -A^{(1,1)}_{j}g^{jm}\pd_{m}\log \paren{ \det g_{\theta} }-A^{(1,1)}_{j}g^{jm}g^{ik}\paren{ \Gamma^{g}_{ik,m} - \Gamma^{g}_{km,i} }  .
  \end{align}
  Then, the condition \eqref{eq:condition_ProbabilityMatchingPrior_gamma_oTF_tensorial} is represented as
  \begin{align}
    \pd_{\gamma}\log \pi + A^{(1,1)}_{i} g^{ij} \pd_{j} \log \pi
     & = \pd_{\gamma}\log c - \pd_{i}A_{j}^{(1,1)}g^{ij}                                                                                                           \\
     & \quad  + A^{(1,1)}_{j}g^{jm} \braces{ \pd_{m}\log c + \pd_{m}\log \paren{  \det g_{\theta} } - g^{ik}\paren{  \Gamma^{g}_{ik,m} - \Gamma^{g}_{km,i}  } }  .
  \end{align}

\subsection{Probability matching prior for the regular parameter}

  Consider the case that $ \theta^{1} $ is the parameter of interest.
  We will calculate the asymptotic expansion of the posterior probability $ P_{\pi}^{n}(U^{1}\leq z) $ and the frequentist probability $ P_{\theta,\gamma}^{n}(U^{1}\leq z) $ to derive the
  conditions of probability matching prior $ \pi^{1}_{PM} $.

  By integrating the expansion of the marginal posterior density \eqref{eq:posterior_density_oTF} over $ t $, we have
  \begin{align}
    \pi\paren{ u;\sample } & = \phi_{d}(u;0,\hat{g}_{\theta}^{-1})
    \braces{ 1+\frac{ 1 }{ \sqrt{n} }\paren{ \frac{ 1 }{ \hat{\pi} }\paren{ D_{\theta}^{\top}\hat{\pi} }  u
        - \frac{1}{\hat{c}}\hat{A}^{(1,1)\top}  u + \frac{ 1 }{ 3! } \hat{A}^{(3,0)\top}  u^{\otimes 3} }+ \bigOp{ \frac{ 1 }{ n } } }
    \label{eq:posterior_distribution_marginal_u}
  \end{align}
  % We decompo
  We denote the density of the $ r $-dimensional normal distribution with mean $ \mu $ and covariance matrix $ \Sigma $ by $ \phi_{r}(\cdot \,;\mu,\Sigma) $.
  Let
  \begin{align}
    u_{-1}       & \coloneqq \paren{ u_{2},\ldots,u_{d} },                                                                                                   \\
    \hat{m}_{-1} & \coloneqq \paren{ \hat{g}_{\theta}^{21}/\hat{g}_{\theta}^{11},\ldots,\hat{g}_{\theta}^{d1} /\hat{g}_{\theta}^{11} }^{\top} \in \bR^{d-1}, \\
    \hat{m}      & \coloneqq \paren{1, \hat{m}_{-1} }^{\top} \in \bR^{d},                                                                                    \\
    \hat{h}_{-1} & \coloneqq\paren{ \hat{g}^{ij}-\hat{g}^{i1}\hat{g}^{j1}/\hat{g}^{11} }_{2\leq i,j\leq d} \in \bR^{(d-1)\times (d-1)},                      \\
    \hat{h}      & \coloneqq  \begin{pmatrix}
                                0      & 0 & \cdots       & 0 \\
                                0      &   &              &   \\
                                \vdots &   & \hat{h}_{-1} &   \\
                                0      &   &              &
                              \end{pmatrix} \in \bR^{d\times d}.
  \end{align}
  We decompose the density $ \phi_{d}(u;0,\hat{g}_{\theta}^{-1}) $ as
  \begin{align}
    \phi_{d}(u;0,\hat{g}_{\theta}^{-1}) & = \phi_{1}(u^{1};0,\hat{g}^{11}) \phi_{d-1}(u_{-1};u^{1}\hat{m}_{-1},\hat{h}_{-1})
  \end{align}
  for the calculation of the probability $ P_{\pi}^{n}(U^{1}\leq z) $.
  Then, the posterior probability of $ U^{1} $ is given by
  \begin{align}
    \pi(u^{1};\sample) & = \int \pi(u;\sample)du_{-1}                                                                                                                                                                                                               \\
                       & = \phi_{1}(u^{1};0,\hat{g}^{11})\int  \braces{ 1+\frac{ 1 }{ \sqrt{n} }\paren{ \frac{ 1 }{ \hat{\pi} }\paren{ D_{\theta}^{\top}\hat{\pi} }  u
    - \frac{1}{\hat{c}}\hat{A}^{(1,1)\top} u  + \frac{ 1 }{ 3! } \hat{A}^{(3,0)\top}  u^{\otimes 3} } }                                                                                                                                                             \\
                       & \quad \cdot \phi_{d-1}(u_{-1};u^{1}\hat{m}_{-1},\hat{h}_{-1})  du_{-1} \quad + \bigOp{ \frac{ 1 }{ n } }                                                                                                                                   \\
                       & = \phi_{1}(u^{1};0,\hat{g}^{11})\left\{ 1  +  \frac{ 1 }{ \sqrt{n} }\left( \paren{ D_{\theta}\log \hat{\pi} }^{\top} \hat{m}u^{1}
    - \frac{ 1 }{ \hat{c} }\hat{A}^{(1,1)\top}\hat{m}u^{1} \right.\right.                                                                                                                                                                                           \\
                       & \quad + \left. \left.\frac{1}{2} \hat{A}^{(3,0)\top}\paren{\mathrm{vec}\, \hat{h}\otimes \hat{m} }u^{1} + \frac{ 1 }{ 3! } \hat{A}^{(3,0)\top} \hat{m}^{\otimes 3}\paren{ u^{1} }^{3}\right)  \right\rbrace  + \bigOp{ \frac{ 1 }{ n } } .
  \end{align}
  Let $ \hat{\sigma}=\sqrt{\hat{g}^{11}} $. The posterior probability of $ U^{1}/\hat{\sigma} $ is
  \begin{align}
    P_{\pi}^{n}(U^{1}/\hat{\sigma} \leq z\mid \sample)
     & = \int_{-\infty}^{ \hat{\sigma}z}\pi( u^{1};\sample)du^{1}                                                                                                                                                                                                     \\
     & = \int_{-\infty}^{ z}\pi( \hat{\sigma}v;\sample) \hat{\sigma} dv                                                                                                                                                                                               \\
     & = \int_{-\infty}^{ z}\phi(v) dv                                                                                                                                                                                                                                \\
     & \quad + \frac{ 1 }{ \sqrt{n} }\braces{ \paren{ D_{\theta}\log \hat{\pi} - \frac{ \hat{A}^{(1,1)} }{ \hat{c} }  }^{\top}\hat{m} +\frac{ 1 }{ 2 } \hat{A}^{(3.0)\top} \paren{ \mathrm{vec}\, \hat{h}\otimes \hat{m} } }\hat{\sigma}\int_{-\infty}^{z} v\phi(v)dv \\
     & \quad + \frac{ 1 }{ 3!\sqrt{n} } \hat{A}^{(3,0)\top} \hat{m}^{\otimes 3}\hat{\sigma}^{3}\int_{-\infty}^{z}v^{3}\phi(v)dv  + \bigOp{ \frac{ 1 }{ n } }                                                                                                          \\
     & = \Phi(z) -  \frac{ 1 }{ \sqrt{n} }\braces{ \paren{ D_{\theta}\log \hat{\pi} - \frac{ \hat{A}^{(1,1)} }{ \hat{c} }  }^{\top}\hat{m} + \frac{ 1 }{ 2 } \hat{A}^{(3.0)\top} \paren{ \mathrm{vec}\, \hat{h}\otimes \hat{m} } }\hat{\sigma} \phi(z)                \\
     & \quad - \frac{ 1 }{ 3!\sqrt{n} } \hat{A}^{(3,0)\top} \hat{m}^{\otimes 3}
    \hat{\sigma}^{3} \paren{ z^{2}+2}  \phi(z)  + \bigOp{ \frac{ 1 }{ n } }                                                                                                                                                                                           \\
     & =\Phi(z) + \frac{ \sqrt{g}^{11}}{ \sqrt{n} }\braces{ -\paren{D_{\theta}\log \pi}^{\top} m  +Q_{3}(\theta,\gamma,z) }\phi(z)+ \bigOp{ \frac{ 1 }{ n } }
    \label{eq:asymptotic_expansion_posterior_u}
  \end{align}
  where
  \begin{align}
    Q_{3}(\theta,\gamma,t)=\frac{ 1 }{ c}A^{(1,1)\top}m - \frac{ 1 }{2} A^{(3,0)\top}\paren{\mathrm{vec}\, h\otimes m } - \frac{ 1 }{ 3! } A^{(3,0)\top} g^{11}m^{\otimes 3}(z^{2}+2) .
  \end{align}
  On the other hands, by \eqref{eq:asymptotic_expansion_posterior_u} and the shrinkage argument,
  we obtain the expansion of the probability $ P_{\theta,\gamma}^{n} (U\leq z) $ as follows:
  \begin{align}
    P_{\theta,\gamma}^{n}(U^{1}/\hat{\sigma}\leq z) & = \Phi_{1}(z)+                                                                \\
                                       & \quad + \frac{ \sqrt{g}^{11}  }{ \sqrt{n} }\paren{D_{\theta}^{\top}\paren{ \sqrt{{g}^{11}}m }
      + Q_{3} (\theta,\gamma,z) }\phi_{1}(z) + \bigO{ \frac{ 1 }{ n } }.
    \label{eq:asymptotic_expansion_frequentist_u}
  \end{align}
  Then, by comparing \eqref{eq:asymptotic_expansion_posterior_u} and \eqref{eq:asymptotic_expansion_frequentist_u},
  we get the conditions for the probability matching prior $ \pmp[1] $ as
  \begin{align}
    \frac{g^{i1}}{\sqrt{{g}^{11}}}\pd_{i}\log \pi & =  -\pd_{i}\paren{ \frac{g^{i1}}{\sqrt{{g}^{11}}} }.
  \end{align}
\section{Proof of Theorem \ref{thm:moment_matching_prior_oTF}}
\label{sec:appendix_pf_moment_matching_prior_oTF}

\subsection{Moment matching prior for the truncation parameter}

By integrating \eqref{eq:posterior_distribution_marginal_t}, the posterior mean of $ t $ is
\begin{align}
  \Exp[t][\sample] & =  \int_{-\infty}^{0} t\e^{t}dt +\frac{ 1 }{ n\hat{c} }\left\lbrace
  \pd_{\gamma}\log \hat{\pi} + D_{\theta}^{\top}\log \hat{\pi}\otimes \hat{A}^{(1,1)\top}  \mathrm{vec} \,\hat{g}_{\theta}^{-1} \right.  \\
                   & \qquad+ \left. \frac{ 1 }{ 2 }\hat{A}^{(2,1)\top}  \mathrm{vec} \,\hat{g}_{\theta}^{-1}
  + \frac{1}{2}\paren{ \hat{A}^{(1,1)} \otimes \hat{A}^{(3,0)} }^{\top}   \mathrm{vec} \paren{\hat{g}_{\theta}^{-1}}^{\otimes{ 2 }}
  \right\rbrace \int_{-\infty}^{0} t(t+1)\e^{t}dt                                                                                        \\
                   & \quad  + \frac{1}{n\hat{c}^{2}} \braces{
                    \frac{ 1 }{ 2 }\hat{A}^{\paren{ 0,2 }} + \frac{1}{2}\paren{ A^{(1,1)\top} }^{\otimes{ 2 }} \mathrm{vec} \,\hat{g}_{\theta}^{-1}
  }\int_{-\infty}^{0} t(t^{2}-2)\e^{t}dt + \bigOp{\frac{1}{n^{3/2}}}                                                                     \\
                   & = -1 +\frac{ 1 }{ n\hat{c} }\left\lbrace
  \pd_{\gamma}\log \hat{\pi} + D_{\theta}^{\top}\log \hat{\pi}\otimes \hat{A}^{(1,1)\top}   \mathrm{vec} \,\hat{g}_{\theta}^{-1} \right. \\
                   & \qquad+ \left. \frac{ 1 }{ 2 }\hat{A}^{(2,1)\top}  \mathrm{vec} \,\hat{g}_{\theta}^{-1}
  + \frac{1}{2}\paren{ \hat{A}^{(1,1)} \otimes \hat{A}^{(3,0)} }^{\top}  \mathrm{vec} \paren{\hat{g}_{\theta}^{-1}}^{\otimes{ 2 }}
  \right\rbrace                                                                                                                          \\
                   & \quad  - \frac{4}{n\hat{c}^{2}} \braces{
  \frac{1}{2}\hat{A}^{\paren{ 0,2 }} + \frac{1}{2}\paren{ \hat{A}^{(1,1)} }^{\otimes{ 2 }\top} \mathrm{vec} \,\hat{g}_{\theta}^{-1}
  }+\bigOp{\frac{1}{n^{3/2}}}.
\end{align}
Then, we have the posterior mean of $ \gamma $ as
\begin{align}
  \hat{\gamma}^{B}_{\pi} = \Exp[\gamma][\sample] & = \cML- \frac{ 1 }{ n\hat{c} }
  +\frac{ 1 }{ n^{2}\hat{c}^{2} }\left\lbrace
  \pd_{\gamma}\log \hat{\pi} + D_{\theta}^{\top}\log \hat{\pi}\otimes \hat{A}^{(1,1)\top}  \mathrm{vec} \,\hat{g}_{\theta}^{-1}\right.     \\
                                                 & \qquad+ \left. \frac{ 1 }{ 2 }\hat{A}^{(2,1)\top}  \mathrm{vec} \,\hat{g}_{\theta}^{-1}
  + \frac{1}{2}\paren{ \hat{A}^{(1,1)} \otimes \hat{A}^{(3,0)} }^{\top}  \mathrm{vec} \paren{\hat{g}_{\theta}^{-1}}^{\otimes{ 2 }}
  \right\rbrace                                                                                                                            \\
                                                 & \quad  - \frac{4}{n^{2}\hat{c}^{3}} \braces{
  \frac{ 1 }{ 2 }\hat{A}^{\paren{ 0,2 }} + \frac{1}{2}\paren{ \hat{A}^{(1,1)\top} }^{\otimes{ 2 }} \mathrm{vec} \,\hat{g}_{\theta}^{-1}
  }+\bigOp{\frac{1}{n^{5/2}}}.
\end{align}
Here, let $ \cML^{*}\coloneqq \cML- \frac{ 1 }{ n\hat{c} } $ be a bias-correrated MLE of $ \gamma $.
Since the consistency of MLE and the low of large numbers, it holds that
\begin{align}
  n^{2}\paren{{\gamma}^{B}_{\pi} - \cML^{*}} & \Pto
  \frac{ 1 }{ {c}^{2} }\left\lbrace
  \pd_{\gamma}\log {\pi} + D_{\theta}^{\top}\log {\pi}\otimes {A}^{(1,1)\top}\mathrm{vec} \,g_{\theta}^{-1}\right.   \\
                                             & \qquad+  \frac{ 1 }{ 2 }{A}^{(2,1)\top} \mathrm{vec}\,g_{\theta}^{-1}
  + \frac{1}{2}{A}^{(1,1)\top} \otimes {A}^{(3,0)\top} \mathrm{vec} \paren{g_{\theta}^{-1}}^{\otimes{ 2 }}
  \\
                                             & \quad  \left.
  - \frac{2}{{c}} {A}^{\paren{ 0,2 }}
  - \frac{2}{{c}} \paren{ A^{(1,1)\top} }^{\otimes{ 2 }} \mathrm{vec} \,g_{\theta}^{-1}\right\rbrace                 \\
                                             & = \frac{ 1 }{ {c}^{2} }\left\lbrace
  \pd_{\gamma}\log {\pi} +  A^{(1,1)\top}g_{\theta}^{-1}D_{\theta}\log \pi \right.                                   \\
                                             & \qquad+  \frac{ 1 }{ 2 }{A}^{(2,1)\top} \mathrm{vec} g_{\theta}^{-1}
  + \frac{1}{2}{A}^{(1,1)\top} \otimes {A}^{(3,0)\top}\mathrm{vec} \paren{ g_{\theta}^{-1} }^{\otimes{ 2 }}
  \\
                                             & \quad  \left.
  - \frac{2}{{c}} {A}^{\paren{ 0,2 }}
  - \frac{2}{{c}} A^{(1,1)\top}g_{\theta}^{-1}A^{(1,1)}\right\rbrace.
\end{align}
The moment matching prior $ \pi_{MM}^{\gamma} $ is required to make the right-hand side of the above equation zero.
Then, the partial differential equation which gives the condition of the moment matching prior $ \pi_{MM}^{\gamma} $ is
\begin{align}
  \pd_{\gamma}\log {\pi} +  A^{(1,1)\top}g_{\theta}^{-1}D_{\theta}\log \pi
   & = -\frac{ 1 }{ 2 }{A}^{(2,1)\top} \mathrm{vec} g_{\theta}^{-1}
  - \frac{1}{2}{A}^{(1,1)\top} \otimes {A}^{(3,0)\top}\mathrm{vec} \paren{ g_{\theta}^{-1} }^{\otimes{ 2 }} \\
   & \quad + \frac{2}{{c}} {A}^{\paren{ 0,2 }} + \frac{2}{{c}} A^{(1,1)\top}g_{\theta}^{-1}A^{(1,1)}.
\end{align}
This condition is also represented as
\begin{align}
  \pd_{\gamma}\log \pi + \frac{1}{2} A^{(2,1)}_{ij}g^{ij}-2\pd_{\gamma}\log c
  + A^{(1,1)}_{i}g^{ij}\braces{ \pd_{j}\log \pi - 2\pd_{j}\log c +\frac{ 1 }{ 2 }A^{(3,0)}_{jkm}g^{km} }=0.
\end{align}

Note that $ A^{(0,2)} = \pd_{\gamma} c$ and $ A^{(1,1)}_{i} = \pd_{i} c $.

\subsection{Moment matching prior for the regular parameter}

By integrating \eqref{eq:posterior_distribution_marginal_u}, the posterior mean of $ u $ is
\begin{align}
  \Exp[u][\sample] & = \int_{\bR^{d}} u\pi\paren{ u;\sample }du                                                                                                           \\
                   & = \int_{\bR^{d}} u\phi_{d}(u;0,\hat{g}_{\theta}^{-1})du
  + \frac{ 1 }{ \sqrt{n} }\int_{\bR^{d}} u\braces{ u^{\top}\frac{ 1 }{ \hat{\pi} }D_{\theta}\hat{\pi}  }\phi_{d}(u;0,\hat{g}_{\theta}^{-1})du                                   \\
                   & \quad - \frac{ 1 }{ \sqrt{n} }\int_{\bR^{d}} u\braces{u^{\top}   \frac{ 1 }{ \hat{c} } \hat{A}^{(1,1)}}\phi_{d}(u;0,\hat{g}_{\theta}^{-1})du               \\
                   & \quad +\frac{ 1 }{ \sqrt{n} }\int_{\bR^{d}} u\braces{ \frac{ 1 }{ 3! } \hat{A}^{(3,0)\top}  u^{\otimes{ 3 }} }\phi_{d}(u;0,\hat{g}_{\theta}^{-1})du
  + \bigOp{ \frac{ 1 }{ n } }                                                                                                                                              \\
                   & = 0+ \frac{ 1 }{ \sqrt{n} }  \hat{g}_{\theta}^{-1}  D_{\theta}\log \hat{\pi}  - \frac{ 1 }{ \sqrt{n} \hat{c}} \hat{g}_{\theta}^{-1}\hat{A}^{(1,1)}   \\
                   & \quad +\frac{ 1 }{ 3! \sqrt{n}} \paren{ I_{d} \otimes  \hat{A}^{(3,0)\top}   } \int_{\bR^{d}}  u^{\otimes{ 4 }} \phi_{d}(u;0,\hat{g}_{\theta}^{-1})du
  + \bigOp{ \frac{ 1 }{ n } }                                                                                                                                              \\
                   & = 0+ \frac{ 1 }{ \sqrt{n} }  \hat{g}_{\theta}^{-1}  D_{\theta}\log \hat{\pi}  - \frac{ 1 }{ \sqrt{n} \hat{c}} \hat{g}_{\theta}^{-1}\hat{A}^{(1,1)}   \\
                   & \quad +\frac{ 3 }{ 3! \sqrt{n}} \paren{ I_{d} \otimes  \hat{A}^{(3,0)\top}   } \symope{4} \mathrm{vec} \paren{ \hat{g}_{\theta}^{-1} }^{\otimes{ 2 }}
  + \bigOp{ \frac{ 1 }{ n } }                                                                                                                                              \\
                   & =  \frac{ 1 }{ \sqrt{n} }  \hat{g}_{\theta}^{-1}  D_{\theta}\log \hat{\pi}  - \frac{ 1 }{ \sqrt{n} \hat{c}} \hat{g}_{\theta}^{-1}\hat{A}^{(1,1)}     \\
                   & \quad +\frac{ 3 }{ 3! \sqrt{n}} \paren{ I_{d} \otimes  \hat{A}^{(3,0)\top}   } \mathrm{vec} \paren{ \hat{g}_{\theta}^{-1} }^{\otimes{ 2 }}
  + \bigOp{ \frac{ 1 }{ n } } .
\end{align}
Since $ \sqrt{n}\paren{ \hat{\theta}_{\pi}^{B} -\qML}=\Exp[u][\sample] $, the consistency of MLE and the low of large numbers, it holds that
\begin{align}
  n \paren{ \hat{\theta}_{\pi}^{B}-\qML }\Pto g_{\theta}^{-1} D_{\theta}\log \pi
  - \frac{ 1 }{ c }g_{\theta}^{-1}A^{(1,1)}+ \frac{ 1 }{ 2 }\paren{ I_{d}\otimes A^{(3,0)\top} }\mathrm{vec}\paren{g_{\theta}^{-1}}^{\otimes{ 2 }}.
\end{align}
Then, the partial differential equation which gives the condition of the moment matching prior $ \pi_{MM}^{\theta} $ is
\begin{align}
  g_{\theta}^{-1} D_{\theta}\log \pi & =
  \frac{ 1 }{ c }g_{\theta}^{-1}A^{(1,1)} - \frac{ 1 }{ 2 }\paren{ I_{d}\otimes A^{(3,0)\top} }\mathrm{vec}\paren{g_{\theta}^{-1}}^{\otimes{ 2 }}.
\end{align}
This condition is also represented as
\begin{align}
  \pd_{j}\log \pi & =  \frac{ 1 }{ 2 }\pd_{j}\log \paren{ \det g } + \frac{ 1 }{ 2 }\aChris[e]{km,j}g_{\theta}^{km}
\end{align}
for $ j=1,\ldots,d $ since $ A^{(3,0)}_{jkm} = - \pd_{j}g_{km} - \aChris[e]{km,j} $ and $ \paren{ \pd_{j}g_{km} }g^{km}=\pd_{j}\log \paren{ \det g_{\theta}  } $.

	\end{appendices}


	\bibliography{sn-bibliography}% common bib file

\end{document}